\documentclass[12pt]{amsart}
\usepackage[margin=1.2in]{geometry}
\usepackage{color}

\title{Block Diagonal Carleson Frames}

\usepackage{hyperref}

\usepackage{amsthm, amsfonts, mathrsfs} 
\usepackage{ dsfont } 
\usepackage{amssymb}
\usepackage{enumerate}
\usepackage{graphicx}
\usepackage{subfig}
\usepackage{physics}
\usepackage{float}
\usepackage{bbm}
\usepackage{amsmath}
\usepackage{comment}
\usepackage{hyperref}
\usepackage{listings}

\usepackage{color}

\usepackage{diagbox}
\usepackage[dvipsnames]{xcolor}
\usepackage{algorithm}
\usepackage{booktabs}
\usepackage{algorithmicx}
\usepackage[noend]{algpseudocode}
\usepackage{tikz}
\usepackage{cite}
\usepackage{orcidlink}
\allowdisplaybreaks

\hypersetup{
    colorlinks,
    citecolor=blue,
    filecolor=blue,
    linkcolor=blue,
    urlcolor=blue
}

\newtheorem{theorem}{Theorem}[section]
\newtheorem{proposition}[theorem]{Proposition}
\newtheorem{lemma}[theorem]{Lemma}
\newtheorem{corollary}[theorem]{Corollary}

\newtheorem{remark}[theorem]{Remark}

\newcommand{\HH}{\mathcal{H}}

\newcommand{\R}{\mathbb{R}}

\newcommand{\N}{\mathbb{N}}
\newcommand{\CC}{\mathbb{C}}

\providecommand{\norm}[1]{\lVert#1\rVert}

\newcommand {\D} {\mathbb D}
\newcommand{\ang}[1]{\langle #1 \rangle}

\newcommand{\Nzero}{\mathbb N_0}
\newcommand{\ellTwo}{\ell^2(\Nzero)}

\newcommand {\la} {\langle}
\newcommand {\ra} {\rangle}
\newcommand{\floor}[1]{\lfloor #1 \rfloor}
\usepackage[foot]{amsaddr}

\newcounter{lst}

\definecolor{ikcolor}{rgb}{1, 0., 0.}

\definecolor{bmcolor}{rgb}{0.9, 0.3, 0}

\begin{document}
\begin{abstract}
    We introduce \textit{block diagonal Carleson frames}, i.e.~singly generated dynamical frames in which the generating operator is block diagonal. We classify all such frames when the generating operator is a direct sum of Jordan blocks, the sizes of which are uniformly bounded. Furthermore, block diagonal Carleson frames are shown to enjoy the high redundancy properties observed for Carleson frames. When the generating operator has nonnegative spectrum, we provide a complete description of the redundancy of such frames under the assumption of the existence of natural density.
\end{abstract}

\author[I. Krishtal]{Ilya Krishtal}
\address{School of Mathematical and Statistical Sciences\\
Northern Illinois University\\
DeKalb, IL 60115, USA}
\email{ikrishtal@niu.edu}

\author[B. Miller]{Brendan Miller}
\address{Department of Mathematics\\
Vanderbilt University\\
1326 Stevenson Center, Nashville, TN 37240, USA}
\email{brendan.miller@vanderbilt.edu}

\maketitle

\section{Introduction}
Born out of questions in sampling theory, there is now a large literature on frames of operator orbits called \textit{dynamical frames} (see the survey \cite{ACKM26} and the references therein). For a given Hilbert space $\mathcal{H}$, the central goal, sometimes referred to as the \textit{dynamical sampling question}, is to classify all triples $(T, \mathcal{G}, \Lambda)$ where $T \in B(\mathcal{H})$, $\mathcal{G \subset \mathcal{H}}$ is countable, and $\Lambda$  is an indexing set such that $\{T^\lambda g\}_{g\in \mathcal{G},\lambda \in \Lambda}$ is a frame for $\HH$. Common references on the subject include \cite{ACMT17} for finite dimensional $\mathcal{H}$ and broad results for infinite dimensions, \cite{CMPP20} for the case of finite generating sets $\mathcal{G}$, \cite{CHP20} for the theory behind singly generated dynamical frames (i.e. $|\mathcal{G}| = 1$), and \cite{ACCMP17} for normal generating operators.   

Among the broader class of dynamical frames are the \textit{Carleson frames}, which have garnered particular interest in recent years. These are singly generated dynamical frames $\{D^k g\}_{k = 0}^\infty$ for which the generating operator $D$ is diagonal with respect to some orthonormal basis. It was originally proved in \cite{ACMT17} that such a sequence is a frame if and only if 1) the point spectrum $\sigma_p(D) = \{z_i\}_{i = 0}^\infty\subset \D$ is a \textit{Carleson sequence}: 
\[
\inf_{k} \prod_{j\neq k} \bigg|\frac{z_j - z_k}{1-\overline{z_j}z_k}\bigg| > 0
\]
and 2) the generating vector $g$ satisfies 
\[
C_1 \sqrt{1-|z_i|^2} \leq |\ang{g,e_i}| \leq C_2\sqrt{1-|z_i|^2}
\]
where $C_2\geq C_1>0$ and $\{e_i\}$ is the set of normalized eigenvectors for $D$. 

While Carleson sequences have been studied for decades (see texts such as \cite{NN2012, Seip04, Garnett81}), Carleson frames have earned the current attention of frame theorists due to their unusually high redundancy. Specifically, the following is known. 

\begin{theorem}\label{T:intro_thm}
    Let $\{D^k g\}_{k = 0}^\infty$ be a Carleson frame, so that $\sigma_p(D) = \{z_i\}_{i = 0}^\infty$ is a Carleson sequence. 
    \begin{enumerate}
        \item[(i)] \cite{CHPS24} If $\{|z_i|\}_{i = 0}^\infty$ is a Carleson sequence, then $\{D^{Nk}g\}_{k = 0}^\infty$ is a frame for any $N \in \N$. 
        \vspace{2mm}
        \item[(ii)]\cite{KM25d} If $\sigma_p(D) \subset (0,1)$, then $\{D^{Nk+j_k} g\}_{k = 0}^\infty$ is a frame for any $N \in \N$ and any $j_k \in [0,N)$. 
    \end{enumerate}
\end{theorem}

The above invites several intriguing follow-up questions, two of which are: do other singly generated dynamical frames admit such redundancy? And what is the extent of the redundancy of such frames when the generating matrix has positive spectrum? The present paper aims to investigate these questions. 

A natural enlargement of the class of Carleson frames is provided by block diagonal generating operators as defined below. They offer a tractable setting in which to study dynamical sampling for genuinely nonnormal systems. Replacing scalar diagonal entries by Jordan blocks allows each spectral point to carry a finite-dimensional family of generalized eigenvectors. On the analytic side, this replaces ordinary point evaluation by simultaneous evaluation of derivatives, so that the frame problem becomes closely connected with stable Hermite interpolation in the Hardy space. With the introduction of block diagonal Carleson frames, we form a bridge between operator-orbit frames, nonnormal spectral theory, and interpolation with multiplicities, while retaining enough structure to permit explicit classification results.

For each integer $i \geq 0$, let $V_i \cong \CC^{n_i}$ be an inner product space with orthonormal basis $\{\delta_{ij}\}_{j = 0}^{n_i-1}$. By $J_{n_i}(z_i)$ denote the operator on $V_i$ which has matrix representation 
\begin{equation}\label{E:jordan_block}
J_{n_i}(z_i)=\begin{bmatrix}
z_i & 0        & 0        & \cdots & 0 \\
1-|z_i|^2       & z_i & 0       & \cdots & 0 \\
0         & 1-|z_i|^2        & z_i & \ddots & \vdots \\
\vdots    & \vdots   & \ddots   & \ddots & 0 \\
0         & 0        & \cdots   & 1-|z_i|^2     & z_i
\end{bmatrix}
\end{equation}
with respect to $\{\delta_{ij}\}_{j = 0}^{n_i-1}$. By a \textit{block diagonal Carleson frame}, or BDC frame, we will mean a singly generated dynamical frame $\{J^k g\}_{k = 0}^\infty$ where $J = \bigoplus\limits_{i = 0}^\infty J_{n_i}(z_i)$ is an operator on $\ell^2 \cong \bigoplus\limits_{i = 0}^\infty V_i$. 

In Section \ref{S: BDCF} we classify all BDC frames by characterizing the admissible generating vectors and showing that such an iterated system forms a frame if and only if the point spectrum is a Carleson sequence.

\begin{theorem}\label{T:characterization}
    Let $J = \bigoplus\limits_{i=0}^\infty J_{n_i}(z_i)$ where $\sup_{i} n_i < \infty$. Then the sequence $\{J^k g\}_{k = 0}^\infty$ is a frame for $\ell^2$ if and only if 
    \begin{enumerate}
        \item $\{z_i\}_{i = 0}^\infty \subset \D$ is a Carleson sequence and 
        \item the vector $g$ is of the form 
        \[
        g = \sum_{i =0}^\infty \sum_{j = 0}^{n_i-1} c_{i,j}\sqrt{1-|z_i|^2}\delta_{i,j}
        \]
        where $\inf\limits_{i} |c_{i,0}| > 0$ and $\sup\limits_{i,j} |c_{i,j}| < \infty$.
    \end{enumerate}
\end{theorem}
\noindent We also provide a generalization of Theorem \ref{T:intro_thm}(\textit{i}) to the block diagonal setting.

Section \ref{S:positive_spec} is devoted to strengthening Theorem \ref{T:intro_thm}(\textit{ii}). Stated for strictly positive spectrum BDC frames, the main theorem of Section \ref{S:positive_spec} is the following.
 \begin{theorem}\label{T:density-carleson}
     Let $\{J^kg\}_{k = 0}^\infty$ be a block diagonal Carleson frame with $\sigma_p(J) \subset (0,1)$. Assume $\Lambda = \{\lambda_k\}_{k = 0}^\infty$ is a strictly increasing sequence where $\lambda_k \in [0,\infty)$ and $\Lambda$ has well-defined natural density: 
     \[
     \lim_{x \to \infty} \frac{N_{\Lambda}(x)}{x} = L
     \]
     for $N_\Lambda(x) = |\{k| \lambda_k \leq x\}|$. Then $\{J^{\lambda_k}g\}_{k = 0}^\infty$ is a frame if and only if $0<L<\infty$. 
 \end{theorem}
 \noindent We will clarify what is meant by noninteger powers of $J$ in Section \ref{S:notation}. 
 
 It is worth mentioning that sequences $\Lambda$ of positive natural density are not necessarily syndetic (or even piecewise syndetic, see \cite{BHR22}). Thus, Theorem \ref{T:density-carleson} is a novel strengthening of Theorem \ref{T:intro_thm}(\textit{ii}) even in the case $J = D$ is a diagonal operator. 

 The redundancy question addressed by the above theorem may also be viewed as a problem of stability under irregular sampling in time. Starting with a frame obtained from observations at every nonnegative integer time, one asks which sets of sampling times still permit stable reconstruction. Our main result shows that, for block diagonal Carleson frames with positive spectrum and for sampling sequences having a well-defined natural density, the answer is governed by the asymptotic sampling rate rather than by the local geometry of the sampling set.

We end the paper in Section \ref{S:completeness} with results regarding the completeness of subsequences of block diagonal Carleson frames. We extend some results from \cite{KM25d} to the case of block diagonal operators, and provide a second proof of Theorem \ref{T:density-carleson} using these completeness results. 

\begin{theorem}\label{T:intro_completeness}
    Let $\{J^kg\}_{k = 0}^\infty$ be a block diagonal Carleson frame with $\sigma_p(J) \subset (0,1)$. The M\"untz condition 
    \[
    \sum_{k = 0}^\infty \frac{\lambda_k}{\lambda_k^2 + 1} = \infty
    \]
    is sufficient, but not necessary for the completeness of sequences of the form $\{J^{\lambda_k} g\}_{k = 0}^\infty$. 
\end{theorem}

The proofs of the above will require several preparatory results. Those that may be useful beyond the current study are gathered in Section \ref{S:prelims}. They include a characterization of the commutant for infinite matrices in Jordan form, a version of the M\"untz-Sz\'asz Theorem for $C^m[0,1]$, and a derivative version of Karamata's Tauberian Theorem. Some more technical lemmas are relegated to the appendices. 

 \subsection{Notation and Conventions}\label{S:notation}
All sequences will be indexed over $\N_0 = \N \cup \{0\}$ for simplicity. As a consequence, it is convenient to index the entries of $n \times n$ matrices $X$ over $ij$ where $0\leq i ,j \leq n-1$. Moreover, the only infinite dimensional Hilbert space we consider is  $\ell^2(\N_0)$. The standard orthonormal basis for $\ell^2(\N_0)$ is denoted by $\{\delta_i\}_{i \in \N_0}$. 

Given a sequence $\{x_i\}_{i \in \N_0}$, by $\{x_i\}_{i \in \N_0} \subset S$ we mean that $x_i \in S$ for all $i \in \N_0$, but not necessarily that the $x_i$ are distinct. For a sequence $\Lambda = \{\lambda_k\}_{k = 0}^\infty$, the function $N_\Lambda(x)$ denotes the counting function of $\Lambda$: 
\begin{equation}\label{E:counting_fnc}
    N_\Lambda(x) = |\{k \mid \lambda_k \leq x\}|.
\end{equation}

By a \textit{Jordan block} we will always mean a matrix of the form \eqref{E:jordan_block} unless otherwise stated. Thus, by a \textit{Jordan chain} associated to a square matrix $X$, we mean a sequence of vectors $v_0,...,v_r$ such that 
\[
Xv_{i} = zv_{i} + (1-|z|^2)v_{i+1}
\]
 for $i < r$ and $X v_{r} = z v_{r}$. The reason for this unorthodox choice of Jordan form will become clear in Section \ref{S: BDCF}. 

 We will need to consider the function space $C^m[a,b]$ of $m$-times continuously differentiable functions on the interval $[a,b]$. In this paper, the norm on this space will always be taken to be
\[
\|f\|_{C^m} = \|f\|_{\infty} + \|f'\|_{\infty}+ \cdots + \|f^{(m)}\|_{\infty}.
\]
For operators, we will always use the notation $\|T\| = \|T\|_{op}$.

Finally, we define noninteger powers of Jordan blocks. For $z\neq 0$, write $J_{n}(z) = zI_{n} + (1-|z|^2)J_{n}(0)$ where $I_n$ is the $n\times n$ identity matrix. Then for $\lambda \in \CC$
\begin{equation}\label{E:frac_powers}
    J^\lambda_{n}(z) := \sum_{j =0 }^{n-1} \binom{\lambda}{j}  z^{\lambda-j} (1-|z|^2)^j J_{n}^j(0)
\end{equation}
where $z^\lambda$ is always defined via the principal branch, and 
\[
\binom{\lambda}{j} = \frac{\lambda(\lambda-1)\cdots (\lambda-j+1)}{j!}
\]
and $\binom{\lambda}{0} = 1$. For $z = 0$,  the operators $J^\lambda_n(0)$ may be defined via \eqref{E:frac_powers} for complex values with $\Re(\lambda) > n-1$. 

It is routine to check that $J_n^{\lambda_1}(z)J_n^{\lambda_2}(z) = J_n^{\lambda_1+\lambda_2}(z)$.

\section{Preliminaries}\label{S:prelims}

\subsection{Infinite Jordan Matrices}\label{s:jordan}

To classify all BDC frames, we will need to understand the commutant of infinite block diagonal matrices $J$ with blocks in Jordan form and of bounded size. This is done in Proposition \ref{P:commutant}.

\begin{lemma}\label{L:norm-est}
    Let $J$ be a finite Jordan block of size $n\times n$. An $n\times n$ matrix $X$ commutes with $J$ if and only if $X$ is a lower-triangular Toeplitz matrix. That is, $X$ has the form
     \[
X= \begin{bmatrix}
c_{0}     & 0      & 0      & \cdots & 0 \\
c_{1}     & c_{0}   & 0      & \cdots & 0 \\
c_{2}     & c_{1}    & c_{0}   & \cdots & 0 \\
\vdots  & \vdots & \vdots & \ddots & \vdots \\
c_{n-1} & c_{n-2} & c_{n-3} & \cdots & c_{0}
\end{bmatrix}.
\]
for some scalars $c_0,...,c_{n-1}$. For such a matrix $X$, we have 
\begin{align}
    \max_{i} |c_i| \leq \|X\| &\leq n\sqrt{n} \max_{i} |c_i| \label{ineq:norm_est},\\
    \|X^{-1}\| &\leq \frac{n^{\frac{n+3}{2}}(\max\limits_i |c_i|)^{n-1} }{|c_0|^n} \label{ineq:inverse_est},
\end{align}
where the estimate on $\|X^{-1}\|$ holds whenever $c_0 \neq 0$.
\end{lemma}

\begin{proof}
    That the commutant of a Jordan block consists of the triangular Toeplitz matrices is a standard linear algebra fact so we omit the proof.  

    The lower bound in \eqref{ineq:norm_est} is derived by observing that the vector $[c_0,...,c_{n-1}]^\top$ is the image of the first standard basis vector under $X$. The upper bound is an easy calculation: 
    \[
\|X\| \leq \sqrt{n} \|X\|_{\infty} \leq  \sqrt{n}\max_{i}\sum_{j = 0}^{n-1} |X_{ij}| \leq n\sqrt n \max_{i} |c_i|.
\] 

To derive \eqref{ineq:inverse_est}, note $X^{-1} = \text{adj}(X)\cdot(1/\det(X)) = \text{adj}(X)/c_0^n$. The $ij$-entry of the matrix $\text{adj}(X)$ is $(-1)^{i+j}M_{ij}$ where $M_{ij}$ is the $(i,j)$-minor of $X$. By Hadamard's inequality, one estimates
\[
|M_{ij}| \leq n^{n/2}(\max_i |c_i|)^{n-1}.
\]
Therefore we conclude that 
\begin{align*}
\|X^{-1}\| &\leq \sqrt{n}\|X^{-1}\|_{\infty} \\
&\leq \ \frac{\sqrt{n} \|\text{adj}(X)\|_{\infty}}{|c_0|^n} \\
&\leq \frac{\sqrt{n} \max\limits_{i} \sum_{j = 0}^{n-1} | M_{ij}|}{|c_0|^n} \\
&\leq \frac{n\sqrt{n}\cdot n^{n/2}(\max\limits_i |c_i|)^{n-1} }{|c_0|^n} \\
&= \frac{n^{\frac{n+3}{2}}(\max\limits_i |c_i|)^{n-1} }{|c_0|^n}
\end{align*}
as claimed.
\end{proof}

\begin{lemma}
    Let $\{z_i\}_{i \in \N_0}\subset \D$ be a Carleson sequence. Consider the operator $J = \bigoplus\limits_{i \in \N_0} J_{n_i}(z_i)$ on $\ell^2(\N_0)$ with $\sup\limits_{i \in \N_0} n_i = n_* < \infty$. If $z_i \neq 0$ for every $i \in \N_0$, then $J^\lambda$ is a bounded, continuously invertible operator on $\ell^2(\N_0)$ for all $\lambda \in \CC$. If $z_0 = 0$, then $J^\lambda$ is bounded if $\lambda \in \N_0$ or $\Re(\lambda) > n_0 - 1$.
\end{lemma}

\begin{proof}
    First assume $z_i \neq 0$ for every $i$, then
    \[
\|J_{n_i}^{\lambda}(z_i)\|
\leq
\sum_{j=0}^{n_*-1}
\left|\binom{\lambda}{j}\right|
|z_i^{\lambda-j}| \leq \sum_{j = 0}^{n_*-1} \bigg| \binom{\lambda}{j}\bigg|e^{\pi|\Im(\lambda)|}|z_i|^{\Re(\lambda)-j}.
\]
    Since $\{z_i\}_{i \in \N_0}$ is a Carleson sequence, we have $\inf\limits_{i \in \N_0} |z_i| = \delta > 0$. Thus we have 
    \[
    \|J^{\lambda}_{n_i}(z_i)\|\leq \sum_{j = 0}^{n_*-1} \bigg| \binom{\lambda}{j}\bigg|e^{\pi|\Im(\lambda)|}|z_i|^{\Re(\lambda)-j} \leq \sum_{j = 0}^{n_*-1} \bigg| \binom{\lambda}{j}\bigg|e^{\pi|\Im(\lambda)|}\max\{1,\delta^{\Re(\lambda) - n_*+1}\}
    \]
    which is uniform over $i$. This shows that $\bigoplus\limits_{i \in \N_0} J^\lambda_{n_i}(z_i)$ is bounded, and applying the same argument to $-\lambda$ shows this operator is invertible. 

    When $z_0 = 0$, $J = J_{n_0}(0) \oplus \bigoplus\limits_{i \geq 1} J_{n_i}(z_i)$. The latter summand is bounded by the preceding argument, so the claim follows since $J_{n_0}^\lambda(0)$ is well defined for $\lambda \in \N_0$ and $\Re(\lambda) > n_0 - 1$.
\end{proof}

\begin{proposition}\label{P:commutant}
    Let $\{z_i\}_{i \in \N_0} \subset \D$ be a Carleson sequence and $\{n_i\}_{i \in \N_0} \subset \N$ a bounded sequence of natural numbers. Let $J = \bigoplus\limits_{i \in \N_0} J_{n_i}(z_i)$. Then a bounded operator $X$ on $\ell^2(\N_0)$ satisfies $XJ = JX$ if and only if $X = \bigoplus\limits_{i \in \N_0} X_i$ where $X_i$ is an $n_i \times n_i$ lower-triangular Toeplitz matrix: 
    \begin{equation}\label{E:inf-toeplitz}
X_i = \begin{bmatrix}
c_{i,0}     & 0      & 0      & \cdots & 0 \\
c_{i,1}     & c_{i,0}   & 0      & \cdots & 0 \\
c_{i,2}     & c_{i,1}    & c_{i,0}   & \cdots & 0 \\
\vdots  & \vdots & \vdots & \ddots & \vdots \\
c_{i,n_i-1} & c_{i,n_i-2} & c_{i,n_i-3} & \cdots & c_{i,0}
\end{bmatrix}
\end{equation}
and $\sup\limits_{i,j} |c_{i,j}| < \infty$. Furthermore, such a bounded matrix $X$ is invertible if and only if $\inf\limits_{i} |c_{i,0}| > 0$.
\end{proposition}

\begin{proof}
    We write $X$ as a block matrix $X = (X_{ij})$ where $X_{ij}$ is a $n_j \times n_i$ matrix. That $XJ = JX$ means that $J_{n_j}(z_j)X_{ij} = X_{ij}J_{n_i}(z_i)$ for all $i,j \in \N_0$. If $i \neq j$, then $z_i \neq z_j$, and the existence and uniqueness theorem for Sylvester equations forces $X_{ij} = 0$. 

    Letting $X_i = X_{ii}$, we thus have $X = \bigoplus\limits_{i \in \N_0} X_{i}$ where $X_iJ_{n_i}(z_i) = J_{n_i}(z_i)X_i$. This forces $X_i$ to have the form \eqref{E:inf-toeplitz} by Lemma \ref{L:norm-est}. Since $\sup\limits_{i\in \N_0} n_i < \infty$, the remainder of the proof now follows from the estimates \eqref{ineq:norm_est} and \eqref{ineq:inverse_est}.
\end{proof}

\subsection{A M\"untz-Type Theorem}\label{muntz} The M\"untz-Sz\'asz Theorem has played a central role in questions regarding dynamical frames (see \cite{KM25d, ACCMP17, ACMT17}). The classical theorem states that a system of monomials $\{x^{\lambda_k}\}_{k \in \N_0}$ is complete in $C[0,1]$ if and only if $\lambda_0 = 0$ and 
\begin{equation}\label{E:muntz-condition}
\sum_{k \in \N_0} \frac{\lambda_k}{\lambda^2_k+1} = \infty.
\end{equation}
On an interval away from the origin, e.g. $[a,b]$ for $a > 0$, \eqref{E:muntz-condition} alone is both necessary and sufficient for the completeness of an arbitrary set of monomials $\{x^{\lambda_k}\}_{k \in \N_0}$ (namely, $\lambda_0 = 0$ is not required) \cite{Almira07}.

For our purposes, we will need the M\"untz-Sz\'asz theorem for $C^m[a,b]$. This follows quite easily from the classical M\"untz theorem on $C[a,b]$ by induction on $m$, and is likely already known to experts. However, we could not find a reference so we include the proof here. 

\begin{proposition}\label{P:muntz_Cm}
    Let $ \{\lambda_k\}_{k=0}^\infty \subset [0,\infty)$ be a strictly increasing sequence with $\lambda_k \to \infty$, and suppose $0 \leq a < b < \infty$. 
    \begin{enumerate}
    \item[(i)] If $a = 0$, then $\{x^{\lambda_k}\}_{k=0}^\infty$ is complete in $C^m[0,b]$ if and only if $\{\lambda_k\}_{k \in \N_0}$ satisfies \eqref{E:muntz-condition}, $\lambda_j = j$ for $0 \leq j \leq m$, and $\lambda_j > m$ for $j \geq m+1$. 
    \vspace{2mm}
        \item[(ii)] If $a > 0$, then the set of monomials $\{x^{\lambda_k}\}_{k=0}^\infty$ is complete in $C^m[a,b]$ if and only if $\{\lambda_k\}_{k \in \N_0}$ satisfies \eqref{E:muntz-condition}. 
    \end{enumerate}
\end{proposition}

\begin{proof}
     We proceed by induction on $m$ and prove both statements at the same time. Assume $\{x^{\lambda_k}\}_{k \in \N_0}$ satisfies \eqref{E:muntz-condition}. If $a = 0$, further assume that $\lambda_j = j$ for $0 \leq j \leq m$ and $\lambda_j > m$ for $j\geq m+1$ . If $a > 0$, then multiplication by $x^{-\lambda_0}$ is a bounded isomorphism of $C^m[a,b]$, so we may always assume $\lambda_0 = 0$. We can further discard finitely many $\lambda_k$ and assume $\lambda_k > 1$ for all $k\geq 1$.

    Let $f \in C^m[a,b]$ be arbitrary. By the induction hypothesis, the set of monomials $\{x^{\lambda_k - 1}\}_{k=1}^\infty$ is dense in $C^{m-1}[a,b]$. Thus, given $\epsilon > 0$, there exists a finite sequence of scalars $c_1,...,c_N$ such that 
    \[
    \|f' - \sum_{k = 1}^N c_k \lambda_k x^{\lambda_k-1}\|_{C^{m-1}} < \epsilon.  
    \]
    Now define 
    \[
    g(x) = f(a) + \sum_{k=1}^N c_k (x^{\lambda_k} - a^{\lambda_k}).
    \]
    Then $g \in \text{span}\{x^{\lambda_k}\}_{k=0}^\infty$, $g(a) = f(a)$, and $\|g' - f'\|_{C^{m-1}} < \epsilon$. 

    We show that $g$ well-approximates $f$ in $C^m[a,b]$. Using that $g(a) = f(a)$, we have 
    \begin{align*}
        \|g - f\|_{C^m} &= \|g-f\|_{\infty} + \|g'-f'\|_{C^{m-1}} \\
        &\leq \sup_{x \in [a,b]} |g(x) - f(x)| + \epsilon \\
        &\leq \sup_{x\in [a,b]} \int_a^x |g'(t) - f'(t)| dt + \epsilon \\
        &\leq \epsilon(b-a+1).
    \end{align*}
    Therefore $\{x^{\lambda_k}\}_{k \in \N_0}$ is complete in $C^m[a,b]$.

    The necessity of the M\"untz condition follows from the classical M\"untz-Sz\'asz Theorem and the standard isomorphism 
    \[
    C^m[a,b] \cong \CC^{m} \oplus C[a,b]
    \]
    given by $f \mapsto (f(a),...,f^{(m-1)}(a),f^{(m)})$.
    Indeed, if the series \eqref{E:muntz-condition} converges, then the set of monomials 
    \[
    \{x^{\lambda_k}\}_{k \in \N_0} \mapsto \{(a^{\lambda_k},...,(m-1)!\binom{\lambda_k}{m-1}a^{\lambda_k-m+1}, m!\binom{\lambda_k}{m}x^{\lambda_k-m})\}_{k \in \N_0}
    \]
    is clearly incomplete in $C^m[a,b]$. This isomorphism also demonstrates that if $a = 0$, we must have $\lambda_j = j$ for $0 \leq j \leq m$ to span the finite dimensional summand $\CC^{m}$.
\end{proof}

\subsection{Karamata's Tauberian Theorem}\label{s:karamata} A hallmark theorem in Tauberian theory, the Hardy-Littlewood-Karamata theorem relates the density of a sequence of points in $[0,\infty)$ with an asymptotic value of the Laplace transform of its counting function. More precisely, let $\Lambda = \{\lambda_k\}_{k = 0}^\infty$ be a strictly increasing sequence of points in $[0,\infty)$ and $N_\Lambda$ its counting function as defined in \eqref{E:counting_fnc}. Adapted to our purposes, the theorem states \cite[Theorem~IV.8.1]{K04t}
\begin{equation}\label{E:classic_Karamata}
\lim_{x \to \infty} \frac{N_\Lambda(x)}{x} = \lim_{t\to 0^+} \;t\int_0^\infty e^{-tx} dN_\Lambda(x) = \lim_{t\to 0^+} t\mathscr{L}[d N_{\Lambda}](t),
\end{equation}
where $\mathscr{L}$ denotes the Laplace transform. It is to be understood that if one of the limits should not exist, then none of the limits exist.

 The proof involves two directions: the first assuming the left limit exists and the other assuming the right limit exists. The former of these two is a direct calculation. We will need the following strengthening of this direct calculation, which includes the asymptotics for the derivatives of the Laplace transform of $dN_\Lambda$. 

 \begin{proposition}\label{P:derivative_Karamata}
    Let $\Lambda = \{\lambda_k\}_{k = 0}^\infty \subset [0,\infty)$ be a strictly increasing sequence with $\lambda_k \to \infty$ and denote by $N_\Lambda$ its counting function. If 
    \[
    \lim_{x \to \infty} \frac{N_\Lambda(x)}{x} = L
    \]
    for some value $L \in [0,\infty)$, then 
    \begin{equation}\label{E:karamata_eq}
     \lim_{t \to 0^+} (-1)^n t^{n+1} \big(\mathscr{L}[dN_\Lambda])^{(n)}(t) = \lim_{t \to 0^+} t^{n+1} \int_0^\infty x^n e^{-tx} dN_\Lambda(x) = L \cdot n!.
    \end{equation}
\end{proposition}

\begin{proof}
        The case of $n=0$ is the classical Karamata theorem, so assume $n \geq 1$. 
Then, using the integration by parts formula for the Stieltjes integral, the substitution formula, and the dominated convergence theorem, we get
        \begin{align*}
           \lim_{t\to 0^+} t^{n+1}\int_0^\infty x^n e^{-tx} dN_\Lambda(x) &= \lim_{t\to 0^+} -t^{n+1}\int_0^\infty N_\Lambda(x) d(x^n e^{-tx}) \\
            &= \lim_{t\to 0^+} -t^{n+1}\int_0^\infty N_\Lambda(x)(nx^{n-1}e^{-tx} - tx^ne^{-tx})dx \\
            &= \lim_{t\to 0^+} \int_0^\infty N_\Lambda(x/t)(tx^n e^{-x}  - ntx^{n-1}e^{-x})dx \\
            &= \int_0^\infty  \lim_{t\to 0^+} \frac{N_\Lambda(x/t)}{x/t}(x^{n+1}e^{-x}  - nx^{n}e^{-x})dx \\
            &= L((n+1)!-n\cdot n!) = Ln!
        \end{align*}
as claimed.
\end{proof}

To make use of this version of Karamata's theorem, we will need to rewrite it in a different form.  

\begin{corollary}\label{C:useful_karamata}
    Assume the hypotheses of Proposition \ref{P:derivative_Karamata}. Then 
    \begin{equation}\label{E:show_instead}
    \lim_{r \to 1^-} j!s!(1-r)^{j+s+1}\sum_{k=0}^\infty \binom{\lambda_k}{j}\binom{\lambda_k}{s} r^{\lambda_k} = L(j+s)!
    \end{equation}
    for every pair $j,s \in \N_0$. As a consequence, 
    \[
    \lim_{r \to 1^-} \frac{\sum_{k} \binom{\lambda_k}{j}\binom{\lambda_k}{s} r^{\lambda_k}}{\sum_{k} \binom{k}{j}\binom{k}{s} r^{k}} = L.
    \]
\end{corollary}

\begin{proof}
    
    We make the change of variable $t  = -\log(r)$ and note that $1-r \asymp -\log(r)$ as $r \to 1^-$. Then 
    \[
    \lim_{r \to 1^-} j!s!(1-r)^{j+s+1}\sum_{k=0}^\infty \binom{\lambda_k}{j}\binom{\lambda_k}{s} r^{\lambda_k} = \lim_{t\to 0^+} j!s!t^{j+s+1} \sum_{k=0}^\infty \binom{\lambda_k}{j}\binom{\lambda_k}{s} e^{-t\lambda_k}. 
    \]
    The coefficients can be written as
    \begin{align*}
    j!s!\binom{\lambda_k}{j}\binom{\lambda_k}{s}&= \big(\lambda_k(\lambda_k-1)\cdots (\lambda_k - j + 1)\big)\big(\lambda_k(\lambda_k-1)\cdots (\lambda_k - s + 1)\big) \\
    &= \lambda_k^n + \sum_{\ell = 0}^{n-1} a_\ell \lambda_{k}^{\ell}
    \end{align*}
    where $n = j+s$ and $a_\ell$ are some coefficients.  

    Thus, 
\begin{align*}
   \lim_{t\to 0^+} j!s!t^{j+s+1} \sum_{k=0}^\infty \binom{\lambda_k}{j}\binom{\lambda_k}{s} e^{-t\lambda_k} &= \lim_{t \to 0^+} t^{n+1}\sum_{k=0}^\infty \lambda_k^n e^{-t\lambda_k} +  \lim_{t\to 0^+}\sum_{\ell=0}^{n-1}  a_\ell t^{n+1} \sum_{k=0}^\infty \lambda_k^{\ell} e^{-t\lambda_k} \\
    &= \lim_{t \to 0^+} t^{n+1}\sum_{k=0}^\infty \lambda_k^n e^{-t\lambda_k} +  \lim_{t\to 0^+}\sum_{\ell=0}^{n-1}  a_\ell t^{n-\ell}\bigg( t^{\ell+1} \sum_{k=0}^\infty \lambda_k^{\ell} e^{-t\lambda_k}\bigg).
\end{align*}
The first term is exactly 
\[
\lim_{t \to 0^+} t^{n+1}\sum_{k=0}^\infty \lambda_k^n e^{-t\lambda_k} = \lim_{t\to 0^+} t^{n+1}\int_0^\infty x^n e^{-tx} dN_\Lambda(x) = Ln!.
\]
The second term satisfies 
\[
\lim_{t\to 0^+}\sum_{\ell=0}^{n-1}  a_\ell t^{n-\ell}\bigg( t^{\ell+1} \sum_{k=0}^\infty \lambda_k^{\ell} e^{-t\lambda_k}\bigg) = \lim_{t\to 0^+}\sum_{\ell=0}^{n-1}  a_\ell t^{n-\ell} \bigg(t^{\ell+1}\int_0^\infty x^\ell e^{-tx} dN_\Lambda(x)\bigg) = 0,
\]
which shows \eqref{E:show_instead}. 
\end{proof}

\section{Characterization and Redundancy of BDC Frames}\label{S: BDCF}
Given a bounded sequence of natural numbers $\{n_i\}_{i\in \N_0} \subset \N$, we have the corresponding decomposition $\ell^2(\N_0) = \bigoplus\limits_{i \in \N_0} V_i$ where $V_i = \text{span} \{\delta_k\}_{k = \sum_{j < i} n_j}^{n_i -1 +\sum_{j < i} n_j} $ for $\delta_j$ the $j^{th}$ standard basis vector in $\ell^2(\N_0)$. Let $\{z_i\}_{i \in \N_0} \subset \D$ be a sequence of complex numbers, and let $J_{n_i}(z_i)$ be the $n_i\times n_i$ matrix given by \eqref{E:jordan_block}. To simplify notation, write $\delta_{i,j} = \delta_{j + \sum_{\ell < i} n_\ell}$ so that $V_i = \text{span}\{\delta_{i,j}\}_{j = 0}^{n_i-1}$.

The first goal of this section is to prove Theorem \ref{T:characterization}. Given an operator $J = \bigoplus\limits_{i\in \N_0} J_{n_i}(z_i)$ as in the theorem, we will call the vector 
\begin{equation}\label{E:canonical_vect}
f = \sum_{i \in \N_0}\sqrt{1-|z_i|^2} \delta_{i,0}
\end{equation}
the \textit{canonical generating vector} associated to $J$. 

The next main result of this section is the promised block diagonal generalization of Theorem \ref{T:intro_thm}(\emph{i}).
\begin{theorem}\label{T:redundancy}
    Let $\{J^k g\}_{k \in \N_0}$ be a block diagonal Carleson frame with $\sigma_p(J) = \{z_i\}_{i \in \N_0} \subset \D$ a Carleson sequence. Assume $\{|z_i|\}_{i \in \N_0}$ is also a Carleson sequence, and let $r > 0$. If $z_i \neq 0$ for all $i\in \N_0$, then $\{J^{rk}g\}_{k \in \N_0}$ is a frame. Otherwise, if $z_0 = 0$ and $K = \floor{\frac{n_0-1}{r}} + 1$, then $\{g,Jg,...,J^{n_0-1}g\} \cup \{J^{rk}g\}_{k = K}^\infty$ is a frame. 
\end{theorem}

\subsection{Characterization}\label{s: vectors} 
To prove Theorem \ref{T:characterization}, we first show that $\{J^k f\}_{k \in \N_0}$ is a frame when $f$ is the canonical generating vector \eqref{E:canonical_vect}. We will need the following result of Nikolskii and Vasyunin. 

\begin{theorem}[\protect{\cite[Lecture~IX]{NN2012}}]
    Let $B(z)$ be a Blaschke product with zeros $\{z_i\}_{i\in \N_0}$ each with multiplicity $n_i$. Assume that $\sup\limits_i n_i < \infty$. Then operator $\Phi : H^2(\D) \to \ell^2$ given by 
    \begin{equation}\label{E:vas_operator}
    \Phi g = \bigg( (1-|z_i|^2)^{s+ \frac{1}{2}} g^{(s)}(z_i) \mid i \in \N_0, 0 \leq s \leq n_i-1  \bigg) 
    \end{equation}
    is bounded and surjective if and only if $\{z_i\}_{i \in \N_0}$ is a Carleson sequence.
\end{theorem}

We will frequently use the natural isomorphism $\ell^2(\N_0) \cong H^2(\D)$ given by 
\[
c = (c_k)_{k \in \N_0} \mapsto \vartheta_c(z) = \sum_{k \in \N_0} c_kz^k.
\]
To show a sequence $\{f_k\}_{k \in \N_0}$ in $\ell^2(\N_0)$ is a frame, it suffices to show the boundedness and surjectivity of its $H^2(\D)$ synthesis operator $\vartheta_c(z) \mapsto \sum_{k} c_k f_k$.

\begin{proposition}\label{T:CV frames}
    Let $\{z_i\}_{i\in \N_0}\subset \D$ be an injective sequence and $\{n_i\}_{i \in \N_0}$ a bounded sequence with $n_i \in \N$. Set $J = \bigoplus\limits_{i\in \N_0} J_{n_i}(z_i)$ and 
    \[
    f = \sum_{i = 0}^\infty \sqrt{1-|z_i|^2} \delta_{i,0}.
    \]
    Then $\{J^k f\}_{k \in \N_0}$ forms a frame if and only if $\{z_i\}_{i \in \N_0}$ is a Carleson sequence. 
\end{proposition}

\begin{proof}
     Since $\sup\limits_{i\in \N_0} n_i < \infty$, the boundedness and surjectivity of the operator \eqref{E:vas_operator} is equivalent to the boundedness and surjectivity of the operator $\Psi$ given by
    \[
    \Psi \vartheta_c = \sum_{i=0}^\infty
\sum_{j=0}^{n_i-1}
\frac{(1-|z_i|^2)^{j+\frac{1}{2}}}{j!}
\vartheta_c^{(j)}(z_i)
\delta_{i,j}.
    \]
    We show that this is the synthesis operator of $\{J^k f\}_{k \in \N_0}$. 

    Let $c \in \ell^2(\N_0)$ be a finitely supported sequence so that $\vartheta_c(z) \in H^2(\D)$ is a polynomial. Let $\widetilde\Psi$ be the $H^2(\D)$ synthesis operator for $\{J^k f\}_{k \in \N_0},$ so that
    \[
    \widetilde\Psi\vartheta_c = \sum_{k = 0}^\infty c_kJ^{k}f.
    \]
    Decomposing $f$, we get 
    \begin{align*}
        \widetilde\Psi \vartheta_c &= \sum_{k = 0}^\infty c_kJ^{k}f \\
        &= \sum_{k = 0}^\infty c_k J^k\sum_{i = 0}^\infty \sqrt{1-|z_i|^2} \delta_{i,0} \\
        &= \sum_{k=0}^\infty \sum_{i=0}^\infty c_k \sqrt{1-|z_i|^2}J^k_i \delta_{i,0}.
    \end{align*}
    Now, a direct calculation shows that 
\begin{equation}\label{E:frame_vectors}
J^k_i\delta_{i,0}
=
\sum_{j = 0}^{n_i-1}
\binom{k}{j}
(1-|z_i|^2)^j
z_i^{k-j}
\delta_{i,j}.
\end{equation}
We are allowed to interchange the order of summation since the sum over $k$ is finite. Thus 
\begin{align*}
\widetilde\Psi \vartheta_c
&=
\sum_{k=0}^\infty
\sum_{i=0}^\infty
c_k \sqrt{1-|z_i|^2}
\sum_{j = 0}^{n_i-1}
\binom{k}{j}
(1-|z_i|^2)^j
z_i^{k-j}
\delta_{i,j} \\
&=
\sum_{i=0}^\infty
\sum_{j = 0}^{n_i-1}
\sum_{k=j}^\infty
c_k
\binom{k}{j}
(1-|z_i|^2)^{j+\frac{1}{2}}
z_i^{k-j}
\delta_{i,j} \\
&=
\sum_{i=0}^\infty
\sum_{j = 0}^{n_i-1}
\frac{1}{j!}
\sum_{k=j}^\infty
c_k
\frac{k!}{(k-j)!}
(1-|z_i|^2)^{j+\frac{1}{2}}
z_i^{k-j}
\delta_{i,j} \\
&=
\sum_{i=0}^\infty
\sum_{j=0}^{n_i-1}
\frac{(1-|z_i|^2)^{j+\frac{1}{2}}}{j!}
\vartheta_c^{(j)}(z_i)
\delta_{i,j} \\
&=
\Psi \vartheta_c ,
\end{align*}
    and the proposition is proved. 
\end{proof}

To complete the proof of Theorem \ref{T:characterization}, we recall the general characterization of generating vectors for singly generated dynamical frames. 

\begin{lemma}[\cite{CHP20}]\label{L:similar}
    Let $T \in B(\ell^2)$ be a bounded operator, and suppose $f \in \ell^2$ such that $\{T^kf\}_{k \in \N_0}$ is a frame. Then $\{T^k g\}_{k \in \N_0}$ for $g \in \ell^2$ is also a frame if and only if there exists an invertible bounded operator $X \in B(\ell^2)$ such that $XTX^{-1} = T$ and $Xf = g$.  
\end{lemma}

\begin{proof}[Proof of Theorem \ref{T:characterization}]
Let $f$ be the canonical generating vector \eqref{E:canonical_vect} so that $\{J^k f\}_{k\in \N_0}$ is a frame. By the lemma, it suffices to characterize all vectors $g = Xf$ where $X$ is an invertible operator with $XJ = JX$. By Proposition \ref{P:commutant}, $X$ is block-diagonal $X = \bigoplus\limits_{i\in \N_0} X_i$ where each $X_i$ is of the form \eqref{E:inf-toeplitz}. It is now easy to check that $Xf$ has the form 
\[
Xf = \sum_{i \in \N_0} \sum_{j = 0}^{n_i-1} c_{i,j}\sqrt{1-|z_i|^2}\delta_{i,j}
\]
where $\sup_{i,j}|c_{i,j}| < \infty$ and $\inf\limits_i |c_{i,0}| > 0$ .
\end{proof}

\subsection{Redundancy of BDC Frames}\label{s:redundancy}

We now turn our attention to Theorem \ref{T:redundancy}. It suffices to prove the theorem for the case $g = f$, the canonical vector, by Lemma \ref{L:similar}. The key is the following technical proposition, the proof of which we move to the Appendix \ref{AP:A} .  

\begin{proposition}\label{P:Bound_COB}
     Let $\{(z_i,n_i)\}  \subset \D\setminus \{0\} \times \N $ be a sequence such that $\{z_i\}$ is a Carleson sequence and $\sup\limits_i{n_i} < \infty$. Then for $r > 0$, there exist $n_i \times n_i$ matrices $L_i$ with $ J_{n_i}(z_i)^r  = L_iJ_{n_i}(z_i^r) L_i^{-1}$ satisfying 
     \[
     \sup\limits_{i}{\{\|L_i\|, \|L_i^{-1}\|\}} < \infty
     \]
     and $L_i\delta_{i,0} = \delta_{i,0}$. Hence, the operator $\bigoplus\limits_{i\in \N_0} L_i \in B(\ell^2)$ is bounded and continuously invertible.
 \end{proposition}

\begin{proof}[Proof of Theorem \ref{T:redundancy}]
    Assume $\{J^k f\}_{k \in \N_0}$ is a frame where $f$ is the canonical generating vector, $\sigma_p(J) = \{z_i\}_{i \in \N_0}$ is such that $\{|z_i|\}_{i \in \N_0}$ is a Carleson sequence. We first consider the case when $z_i \neq 0$ for all $i$. 
    
    Let $L = \bigoplus\limits_{i\in \N_0} L_i$ where $L_i$ are the matrices in Proposition \ref{P:Bound_COB}. Then  $L$ is a bounded invertible operator and  $L^{-1}J^rL = \bigoplus_i J_{n_i}(z_i^r) =: J_r$. Since $L_i\delta_{i,0} = \delta_{i,0}$, we have $Lf = f$. It therefore suffices to show that $\{J_r^kf\}_{k=0}^\infty$ is a frame. 

    Now, it has been shown in \cite{CHPS24} that a Carleson sequence $\{z_i\}_{i \in \N_0}$ satisfying the hypotheses has the property that $\{z_i^r\}_{i \in \N_0}$ is also a Carleson sequence (note that in the paper cited, this was only shown for $r \in \N$, but the proof is easily generalizable). By Theorem \ref{T:characterization}, it will follow that $f$ is an admissible generating vector for $J_r$ if
    \[
    0 <\inf_i \frac{\sqrt{1-|z_i|^2}}{\sqrt{1-|z_i^r|^2}} \leq \sup_i \frac{\sqrt{1-|z_i|^2}}{\sqrt{1-|z_i^r|^2}} < \infty.
    \]
    But this is clear since 
    \[
    \min\left\{\frac{1}{r},1\right\} \leq \frac{1-x^2}{1-x^{2r}} \leq \max\left\{\frac{1}{r},1\right\}.
    \]

    Now suppose $z_0 = 0$ and set $K = \floor{\frac{n_0-1}{r}} + 1$. Then $\overline{\text{span}}\{J^{rk}f\}_{k = K}^\infty \subset \bigoplus\limits_{i\geq 1} V_i$, and $\{f,Jf,...,J^{n_0-1}f\}$ is an outer frame for $V_0$. Since $J$ is invertible on $\bigoplus\limits_{i\geq 1}V_i$, the preceding argument shows that $\{J^{rk}f\}_{k = K}^\infty$ is in fact a frame for $\bigoplus\limits_{i\geq 1}V_i$. It follows that $\{f,Jf,...,J^{n_0-1}f\} \cup \{J^{rk}f\}_{k = K}^\infty$ is a frame for $\ell^2(\N_0)$. 
\end{proof}

\section{Generators with Nonnegative Spectrum}\label{S:positive_spec}

In this section, we study BDC frames for which the generating matrix has nonnegative real spectrum. We provide a complete classification of their redundancy under the existence of a natural density. We will prove the following result, from which Theorem \ref{T:density-carleson} will follow as a special case. 

\begin{theorem}\label{T:Density-redundancy}
    Let $\{J^k g\}_{k \in \N_0}$ be a block diagonal Carleson frame with $\sigma_p(J) = \{z_j\}_{j \in \N_0} \subset [0,1)$ a Carleson sequence. Suppose $\Lambda = \{\lambda_k\}_{k = 0}^\infty \subset[0,\infty)$ is a strictly increasing sequence satisfying 
    \begin{equation}\label{E:limit_assump}
    \lim_{x \to \infty} \frac{N_\Lambda(x)}{x} = L
    \end{equation}
    where $N_\Lambda$ is the counting function associated to $\Lambda$. 
    \begin{enumerate}
        \item[(i)] If J is invertible, then $\{J^{\lambda_k} g\}_{k \in \N_0}$ is a frame if and only if $0 < L < \infty$. 
        \item[(ii)] If $0 \in \sigma_p(J)$, then $\{J^{\lambda_k} g\}_{k \in \N_0}$ is a frame if and only if $0 < L < \infty$ and $\lambda_j =j$ for all $0 \leq j \leq n_0 -1$ where $n_0$ is the algebraic multiplicity of $0$ as an eigenvalue of $J$.
    \end{enumerate} 
\end{theorem}

\begin{remark}
 \emph{The proof of the ``if" direction of the above theorem relies heavily on the natural density of $\Lambda$ being well defined, i.e.~that the limit \eqref{E:limit_assump} exists. That is, the proof will fail (although the statement may remain true) if we relax the hypothesis on $\Lambda$ to }
\begin{equation}\label{E:relaxed}
 0 < \liminf_{x\to \infty} \frac{N_\Lambda(x)}{x} \leq \limsup_{x\to \infty}\frac{N_\Lambda(x)}{x} < \infty .
\end{equation}
\emph{
    \noindent But, it is clear that the condition that the sequence $\Lambda$ has a well-defined natural density may be replaced by the requirement that there exists $\Gamma \subseteq \Lambda$, which has a well-defined natural density (assuming that the system remains Bessel). This can be characterized using the notion of the \emph{macroscopic lower interval density}
    \[
    d^-_{\rm{mac}}(\Lambda) := \inf_{\eta \in (0,1)}\liminf_{x\to \infty}\inf_{0\le y\le(1-\eta)x}\frac{N_\Lambda(x) - N_\Lambda(y)}{x-y}.
    \]
    It is not hard to show that 
    \[
    \Lambda \mbox{ contains } \Gamma \subseteq \Lambda \mbox{ with } \lim_{x\to \infty}\frac{N_\Gamma(x)}{x} = L
    \]
    if and only if $d^-_{\rm{mac}}(\Lambda) \ge L$.}
\end{remark}

We will fix notation for the remainder of this section. Let $\{z_i\}_{i \in \N_0} \subset [0,1)$ be an increasing Carleson sequence of nonnegative real numbers and $\{n_i\}_{i \in \N_0} \subset \N$ a bounded sequence and $n_{*} = \max\limits_{i\in \N_0} n_i$. We have the corresponding decomposition $\ell^2(\N_0) = \bigoplus\limits_{i \in \N_0} V_i$ where $V_i = \overline{\text{span}}\{\delta_{i,j}\}_{j = 0}^{n_i-1}$ as in Section \ref{S: BDCF}. 

Denote by $J$ the operator 
\[
J = \bigoplus_{i \in \N_0} J_{n_i}(z_i)
\]
as before so that each $J_{n_i}(z_i)$ acts on $V_i$. We only need to prove Theorem \ref{T:Density-redundancy} for the canonical generating vector $f$, so we fix 
\[
f = \sum_{i = 0}^\infty \sqrt{1-z_i^2}\delta_{i,0}.
\]

Finally, we will prove both cases of Theorem \ref{T:Density-redundancy} at the same time, similar to the proof of Proposition \ref{P:muntz_Cm}. By $\Lambda = \{\lambda_k\}_{k=0}^\infty$ we denote a strictly increasing sequence of real numbers with $\lambda_k \to \infty$. If $z_0 =  0 \in \sigma_p(J)$, then we always assume $\lambda_j = j$ for $0\leq j \leq n_0-1 $. The frame operator associated to $\{J^{\lambda_k} f\}_{k \in \N_0}$ is denoted by $S_\Lambda$.  

\subsection{Besselness and Necessity Conditions}
We first establish that $\{J^{\lambda_k}g\}_{k \in \N_0}$ is Bessel whenever $\Lambda$ has finite natural density. This will allow us to write a matrix representation for the frame operator $S_\Lambda$, with which we prove the necessity conditions of Theorem \ref{T:Density-redundancy}. 

\begin{proposition}\label{P:besselness}
Assume the hypotheses of Theorem \ref{T:Density-redundancy} and suppose that
\[
    \lim_{x\to\infty}\frac{N_\Lambda(x)}{x}<\infty .
\]
Then \(\{J^{\lambda_k}f\}_{k\in\mathbb N_0}\) is Bessel.
\end{proposition}

\begin{proof}
It is enough to prove the claim for the case when $0 \notin \sigma_p(J)$ and $\lambda_k > n_*-1$ for all $k\in \N_0$. 

Let $b \in \ell^2(\N_0)$. Then 
\begin{equation}\label{E:expansion}
\ang{b,J^{\lambda_k}f} = \sum_{i\in\mathbb N_0}
    \sum_{j=0}^{n_i-1}
   b_{i,j}
    \binom{\lambda_k}{j}
    (1-z_i^2)^{j+\frac12}
    z_i^{\lambda_k-j}.
\end{equation}
Therefore 
\begin{align*}
\sum_{k=0}^\infty |\langle b,J^{\lambda_k}f\rangle|^2
&=
\sum_{k=0}^\infty
\left|
\sum_{i\in\N_0}\sum_{j=0}^{n_i-1}
b_{i,j}
\binom{\lambda_k}{j}
(1-z_i^2)^{j+\frac12}
z_i^{\lambda_k-j}
\right|^2 \\
&=
\sum_{i,r\in\N_0}
\sum_{j=0}^{n_i-1}
\sum_{s=0}^{n_r-1}
b_{i,j}\overline{b_{r,s}}
\frac{(1-z_i^2)^{j+\frac12}
(1-z_r^2)^{s+\frac12}}{
z_i^{j}z_r^{s}}
\sum_{k=0}^\infty
\binom{\lambda_k}{j}
\binom{\lambda_k}{s}
(z_i z_r)^{\lambda_k}.
\end{align*}
Since $\lambda_k > n_*-1$, the value of the above can only increase if we replace $b$ by its coordinate-wise absolute value. So, assume $b_{i,j} \geq 0$ for all $i,j$. 

Using Karamata's theorem (see Corollary \ref{C:useful_karamata}), there exists a $C > 0$ such that 
\[
\sum_{k=0}^\infty
\binom{\lambda_k}{j}
\binom{\lambda_k}{s}
(z_i z_r)^{\lambda_k} \leq C\sum_{k=0}^\infty
\binom{k}{j}
\binom{k}{s}
(z_i z_r)^{k}. 
\]
Hence, 
\begin{align*}
\sum_{k \in \N_0} |\ang{b,J^{\lambda_k}f}|^2 &\leq C \sum_{i,r\in\N_0}
\sum_{j=0}^{n_i-1}
\sum_{s=0}^{n_r-1}
b_{i,j}b_{r,s}
\frac{(1-z_i^2)^{j+\frac12}
(1-z_r^2)^{s+\frac12}}{
z_i^{j}z_r^{s}}
\sum_{k=0}^\infty
\binom{k}{j}
\binom{k}{s}
(z_i z_r)^{k} \\
&= C\sum_{k \in \N_0} |\ang{b, J^k f}|^2,
\end{align*}
from which the Bessel bound follows since $\{J^k f\}_{k \in \N_0}$ is Bessel.
\end{proof}

\begin{lemma}\label{L:frame_operator}
    The frame operator $S_\Lambda$ associated to $\{J^{\lambda_k} f\}_{k \in \N_0}$ has the matrix representation 
    \[
    \ang{S_\Lambda \delta_{i,j}, \delta_{r,s}} = (1-z_i^2)^{j+\frac{1}{2}}(1-z_r^2)^{s+\frac{1}{2}}\sum_{k \in \N_0} \binom{\lambda_k}{j}\binom{\lambda_k}{s} z_i^{\lambda_k - j} z_r^{\lambda_k - s}. 
    \]
\end{lemma}

\begin{proof}
    This is a direct calculation: 
    \begin{align*}
    \ang{S_{\Lambda}\delta_{i,j}, \delta_{r,s}} &=  \bigg \langle \sum_{k = 0}^\infty\ang{J^{\lambda_k}f, \delta_{i,j}} J^{\lambda_k}f, \delta_{r,s}\bigg \rangle \\
    &= \sum_{k \in \N_0} \ang{J^{\lambda_k} f,\delta_{i,j}}\ang{J^{\lambda_k}f,\delta_{r,s}} \\
    &= (1-z_i^2)^{j+\frac{1}{2}}(1-z_r^2)^{s+\frac{1}{2}}\sum_{k \in \N_0} \binom{\lambda_k}{j}\binom{\lambda_k}{s} z_i^{\lambda_k - j} z_r^{\lambda_k - s}
    \end{align*}
as claimed.
\end{proof}

The ``only if" direction of Theorem \ref{T:Density-redundancy} was recently explored for Carleson frames in \cite{GP26arXiv} (absent the natural density hypothesis). The proof of the following proposition, which establishes the ``only if" direction of the Theorem, is nearly identical to their calculations.
\begin{proposition}
    Assume the hypotheses of Theorem \ref{T:Density-redundancy} and that $\{J^{\lambda_k}f\}_{k \in \N_0}$ is a frame. Then 
    \[
    0 <  \lim_{x \to \infty} \frac{N_\Lambda(x)}{x} < \infty. 
    \]
    Furthermore, if $0 \in \sigma_p(J)$ then $\lambda_j = j$ for all $0 \leq j \leq n_0-1$. 
\end{proposition}

\begin{proof}
    Let 
    \[
    L = \lim_{x\to \infty} \frac{N_\Lambda(x)}{x}.
    \]
    If $L = 0$, then 
    \[
    \ang{S_\Lambda \delta_{i,0}, \delta_{i,0}} = (1-z_i^2)\sum_{k \in \N_0} z_i^{2\lambda_k} \to L = 0
    \]
    as $i \to \infty$ by Karamata's theorem \eqref{E:show_instead}. Hence $\{J^{\lambda_k}f\}_{k \in \N_0}$ is not a frame, a contradiction. 

    If $L = \infty$, then again we have 
    \[
\ang{S_\Lambda \delta_{i,0}, \delta_{i,0}} = (1-z_i^2)\sum_{k \in \N_0} z_i^{2\lambda_k} \to L = \infty
\]
so $\{J^{\lambda_k}f\}_{k \in \N_0}$ is not Bessel. 

Finally, assume that $z_0 = 0 \in \sigma_p(J)$. Then $J_{n_0}(0)$ is nilpotent, and $P_{V_0}J^\lambda P_{V_0} = J^\lambda_{n_0}(0) = 0$ for all $\lambda > n_0-1$. Since $\dim V_0 = n_0$ and $J^\lambda_{n_0}(0)$ is not well defined for noninteger values $0 \leq \lambda \leq n_0-1$, it is clear that the image of $\overline{\text{span}}\{J^{\lambda_k}f\}$ under the projection $P_{V_0}$ spans $V_0$ only if $\lambda_j = j$ for all $0 \leq j \leq n_0 - 1$. 
\end{proof}

\subsection{Proof of Sufficiency Conditions}  The proof strategy employed here follows the same outline as that of \cite[Theorem~2.3]{KM25d}. 

The following Proposition uses the M\"untz theorem for $C^m$ to show that $\{J^{\lambda_k} g\}_{k \in \N_0}$ is a frame if it is an outer frame for a suitable subspace of finite codimension. This is a direct generalization of \cite[Lemma~2.2]{KM25d}, and its proof is similar, but more technical. So, the proof is postponed to Appendix \ref{AP:B}. 

\begin{proposition}\label{P:outer_frame_to_frame}
     Assume the hypotheses of Theorem \ref{T:Density-redundancy}. If there exists some $M \in \N_0$ such that $\{J^{\lambda_k} f\}_{k \in \N_0}$ is an outer frame for $W_M = \overline{\text{span}}\{\delta_{i,j}\}_{i \geq M} = \bigoplus_{i \geq M} V_i$, then $\{J^{\lambda_k}f\}_{k \in \N_0}$ is a frame.
\end{proposition}


Another technical result needed to prove Theorem \ref{T:Density-redundancy} is the following. Once again, the proof is contained in Appendix \ref{AP:C}. 

\begin{lemma}\label{L:frame_operator2}
    The frame operator $S_{\N_0}$ has uniformly bounded summable rows and columns. More precisely, there is a constant $c$ depending only on $\sigma_p(J)$ and a constant $C_{n_{*}}$ depending only on $n_{*}$ such that
    \begin{equation}\label{E:frame_operator_estimate}
    \sum_{i\in \N_0} \sum_{j = 0}^{n_i-1} |\ang{S_{\N_0} \delta_{i,j}, \delta_{r,s}}| \leq \frac{C_{n_{*}}}{(1-c^{1/2})}.
    \end{equation}
    for all $r \in \N_0, \;0\leq s \leq n_r-1$.
\end{lemma}

The final preparatory result is the following easy perturbation lemma. 
\begin{lemma}\label{L:PD_op}
    Let $V \in B(\ell^2)$ be a positive definite operator with bounds $B \geq A > 0$ so that
    \[
    A \|c\|^2\leq \ang{Vc,c} \leq B\|c\|^2 
    \]
    for all $c \in \ell^2$. Further, let $a> 0$ be a positive number and $W $ be a bounded symmetric operator on $\ell^2$. If $\|W\|< a\cdot A$, then 
    \[
    S = aV + W 
    \]
    is a bounded positive definite operator. 
\end{lemma}

\begin{proof}
    Since $W$ is symmetric, we have 
    \[
    |\ang{Wc,c}| \leq \|W\|\cdot \|c\|^2
     \]
     for all $c \in \ell^2$. Thus, we can compute directly that 
     \begin{align*}
         \ang{Sc,c} &= \ang{aVc,c} + \ang{Wc,c} \\
         &\geq a\cdot A\|c\|^2 - \|W\|\cdot \|c\|^2 \\
         &\geq (a\cdot A-\|W\|)\cdot \|c\|^2.
     \end{align*}
     By assumption, $a\cdot A-\|W\| > 0$. One similarly shows that 
     \[
     \ang{Sc,c} \leq (a\cdot B+\|W\|)\cdot \|c\|^2,
     \]
     hence we have that 
     \[
     0< (a\cdot A-\|W\|)\cdot \|c\|^2\leq \ang{Sc,c} \leq (a\cdot B+\|W\|)\cdot \|c\|^2,
     \]
     which completes the proof.
\end{proof}

\begin{proof}[Proof of Theorem \ref{T:Density-redundancy}]
    Assume that $0 < L < \infty$ and let $P_M$ denote the orthogonal projection onto $W_M = \bigoplus_{i \geq M} V_i$. By Proposition \ref{P:outer_frame_to_frame} it suffices to show that $\{P_{M} J^{\lambda_k} f\}_{k \in \N_0}$ is a frame for $W_M$ for some $M\geq 0$. We show that the frame operator $P_MS_\Lambda P_M$ is positive definite. 

    Let $A$ denote the lower frame bound for $\{J^{k}f\}_{k \in \N_0}$. Then $A$ is also a lower frame bound for $P_M S_{\N_0}P_M$ for all $M\geq 0$. If we decompose 
    \[
    P_MS_\Lambda P_M = LP_MS_{\N_0}P_M + P_M(S_\Lambda - LS_{\N_0})P_M,
    \]
    then by Lemma \ref{L:PD_op}, we only need to show that $\|P_M(S_\Lambda - LS_{\N_0})P_M\| < L A$ for sufficiently large $M$. To this end, let $\epsilon > 0$ so that 
    \[
    \epsilon < \frac{LA(1-c^{1/2})}{C_{n_{*}}}
    \]
    where $c$ and $C_{n_{*}}$ are as in Lemma \ref{L:frame_operator2}.
    
    By Lemma \ref{L:frame_operator} and Corollary \ref{C:useful_karamata}, we have 
    \begin{align*}
    \frac{(S_\Lambda)_{(i,j),(r,s)}}{(S_{\N_0})_{(i,j),(r,s)}} &= \frac{\sum_{k} \binom{\lambda_k}{j}\binom{\lambda_k}{s} (z_iz_r)^{\lambda_k}}{\sum_{k} \binom{k}{j}\binom{k}{s} (z_iz_r)^{k}} \to L.
    \end{align*}
    So, we may choose $M \gg 0$ so that 
    \[
    |(S_\Lambda)_{(i,j),(r,s)} - L(S_{\N_0})_{(i,j),(r,s)}| < \epsilon |(S_{\N_0})_{(i,j),(r,s)} |
    \]
    for all $i,r \geq M$ and $0 \leq j \leq n_i-1, 0 \leq s \leq n_r-1$. 
    Thus, using \eqref{E:frame_operator_estimate} we can bound the row sums (and hence the column sums) of $P_M(S_\Lambda - LS_{\N_0})P_M$ as
    \begin{equation}\label{E:out_frame_calc}
    \begin{split}
        \sum_{i\in \N_0} \sum_{j = 0}^{n_i-1}  |(P_M(S_\Lambda - LS_{\N_0})P_M)_{(i,j),(r,s)}| &\leq \sum_{i \in \N_0} \sum_{j = 0}^{n_i-1} \epsilon|(S_{\N_0})_{(i,j),(r,s)}| \\
        &\leq \epsilon \frac{C_{n_{*}}}{1-c^{1/2}} \\
        &< LA.
        \end{split}
    \end{equation}
By Schur's test, this implies $\|P_M(S_\Lambda - LS_{\N_0})P_M\| < LA$, which completes the proof of sufficiency.
\end{proof}

\section{Completeness of Operator Orbits}\label{S:completeness}
It was recently conjectured in \cite{ACKM26} that if $\{D^{k}g\}_{k \in \N_0}$ is a Carleson frame with $\sigma_p(D) \subset (0,1)$, then $\{D^{\lambda_k}g\}_{k \in \N_0}$ is again a frame if and only if it is Bessel and $\Lambda = \{\lambda_k\}_{k \in \N_0}$ satisfies the M\"untz condition. The authors in \cite{GP26arXiv} show that while the M\"untz condition is necessary, it is not sufficient to guarantee a lower frame bound (this also follows from Theorem \ref{T:Density-redundancy}). On the other hand, it is known that the M\"untz condition is sufficient for \textit{completeness} of such sequences (see \cite[Theorem~3.6]{KM25d}). So, it is tempting to restate the conjecture as: \textit{$\{D^{\lambda_k}g\}_{k\in \N_0}$ is complete if and only if $\{\lambda_k\}_{k \in \N_0}$ satisfies the M\"untz condition}. We will show that this is again false. 

In this section, we begin an investigation into the completeness of sampled orbits $\{J^{\lambda_k}g\}_{k \in \N_0}$ where $\{J^{k}g\}_{k \in \N_0}$ is a positive spectrum BDC frame. We show that the results proved in \cite{KM25d} for Carleson frames carry over to the block diagonal setting. In particular, we will prove Theorem \ref{T:intro_completeness}, and offer an alternative proof of Theorem \ref{T:Density-redundancy}. As before, we may always assume that $g = f$ is the canonical generating vector for the results in this section.

\subsection{Alternate Proof of Theorem \ref{T:Density-redundancy}}

We begin with a completeness result that generalizes \cite[Theorem~3.5]{KM25d}.

\begin{theorem}\label{T:sector_completeness}
    Let $\{J^k f\}_{k \in \N_0}$ be a block diagonal Carleson frame in which 
    \[
    \sigma_p(J) \subset \D_c = \{re^{i\theta} \mid r \in (0,1), \theta \in [-c,c]\} 
    \]
    for some $0 \leq c < \pi$. If $\Lambda = \{\lambda_k\}_{k \in \N_0} \subset [0,\infty)$ is a strictly increasing sequence with  
    \begin{equation}\label{E:log_density}
    L(\Lambda) := \inf_{\mu > 1} \frac{1}{\log \mu} \limsup_{t \to \infty} \sum_{\Lambda \cap [t,\mu t]}\frac{1}{\lambda} > \frac{c}{\pi}.
    \end{equation}
    Then $\{J^{\lambda_k}f\}_{k \in \N_0}$ is complete in $\ell^2(\N_0)$. 
\end{theorem}

\begin{proof}
    The proof is nearly identical to that of \cite[Theorem~3.5]{KM25d}, so we will omit some details. 

    Let $0 \neq b \in \ell^2(\N_0)$ and consider the complex function $\psi(\omega) = \ang{ J^{\omega}f,b}$. Expanding as in \eqref{E:expansion}, one checks that $\psi$ is an entire function of exponential type. In particular, $\psi$ is not identically zero since $\{J^kf\}_{k \in \N_0}$ is a frame. 
    
    On the imaginary axis,
\[
       |\overline z_i^{\,iy}|\le e^{c|y|}
\]
for $z_i \in \sigma_p(J)$. For $0 \leq q\leq n_*-1$,
\[
       \left|\binom{iy}{q}\right|\leq C(1+|y|)^{n_*-1}.
\]
Consequently
\[
       |\psi(iy)|\leq C(1+|y|)^{n_*-1}e^{c|y|}.
\]
For each $\varepsilon>0$, the polynomial factor is bounded by a constant times $e^{\varepsilon|y|}$. 

By Rubel's generalization of Carlson's Theorem \cite[\S 3]{R56}, it follows that the zero set of $\psi$ satisfies $L(\psi^{-1}(0)) \leq \frac{c}{\pi}$. Therefore $\{J^{\lambda_k}f\}_{k \in \N_0}$ is complete if $L(\Lambda) > \frac{c}{\pi}$.
\end{proof}
\begin{remark}\label{R:LBD}
    \emph{The value $L(\Lambda)$ in \eqref{E:log_density} is called the \textit{logarithmic block density} of $\Lambda$. A sequence $\Lambda$ which has well-defined, positive natural density must also have positive logarithmic block density. In fact, these values are equal in this case. We offer an informal explanation.}

    \emph{If $\Lambda = \{\lambda_k\}_{k \in \N_0}$ has well-defined natural density, then 
    \[
    L =\lim_{x \to \infty} \frac{N_\Lambda(x)}{x} = \lim_{k \to \infty} \frac{k}{\lambda_k}.
     \]
     Hence, $\frac{1}{\lambda_k} \approx \frac{L}{k}$ for $k\gg 0$. So for $t \gg 0$
     \begin{align*}
     \sum_{\lambda_k \in \Lambda\cap [t,\mu t]} \frac{1}{\lambda_k} &\approx \sum_{k = N_{\Lambda}(t)}^{N_\Lambda(\mu t)} \frac{L}{k} \\
     &\approx L \log \frac{N_\Lambda(\mu t)}{N_\Lambda(t)} \\
     &\approx L \log\mu.
     \end{align*}
     Therefore, 
     \[
     \inf_{\mu > 1} \frac{1}{\log \mu} \limsup_{t \to \infty} \sum_{\Lambda \cap [t,\mu t]}\frac{1}{\lambda} = \inf_{\mu > 1} \frac{1}{\log \mu} \cdot L\log{\mu} = L.
     \]}
\end{remark}

Just as in \cite[Theorem 3.6]{KM25d}, in the case of $\sigma_p(J) \subset (0,1)$, the logarithmic block density condition may be replaced by the M\"untz condition.

\begin{proposition}\label{P:muntz_sufficient}
    Let $\{J^kf\}_{k \in \N_0}$ be a block diagonal Carleson frame with $\sigma_p(J) \subset (0,1)$. Suppose $\{\lambda_k\}_{k \in \N_0} \subset [0,\infty)$ is a strictly increasing sequence satisfying 
    \[
    \sum_{k = 0}^\infty \frac{\lambda_k}{\lambda_k^2+1} = \infty.
    \]
    Then $\{J^{\lambda_k}f\}$ is complete. 
\end{proposition}

\begin{proof}
    Assume first that $\lambda_k \to \infty$. Let $\sigma_p(J) = \{z_i\}_{i \in \N_0}$ with $z_i < z_{i+1}$. If $\{J^{\lambda_k}f\}_{k \in \N_0}$ is incomplete, then there exists a $0\neq b \in \ell^2(\N_0)$ such that 
    \begin{align*}
        0 &= \ang{J^{\lambda_k}f,b} \\
        &= \sum_{i\in\mathbb N_0}
    \sum_{j=0}^{n_i-1}
   \overline{b_{i,j}}
    \binom{\lambda_k}{j}
    (1-z_i^2)^{j+\frac12}
    z_i^{\lambda_k-j}
    \end{align*}
    for all $\lambda_k$. This formula extends naturally to a bounded functional on $C^m[z_0,1]$ (where $m = n_*-1$) via 
    \[
    h \mapsto \sum_{i \in \N_0} \sum_{j = 0}^{n_i-1} \overline{b_{i,j}} \frac{(1-z_i^2)^{j+\frac{1}{2}}}{j!} h^{(j)}(z_i). 
    \]
    It is easy to show that this functional is nonzero (using Lemma \ref{L:construction}, for example). This functional then annihilates all the monomials $x^{\lambda_k}$, but this is impossible by Proposition \ref{P:muntz_Cm}. Therefore $\{J^{\lambda_k}f\}_{k \in \N_0}$ is complete.

    If $\{\lambda_k\}_{k \in \N_0}$ has a finite limit point, and $\ang{J^{\lambda_k}f,b} = 0$ for all $\lambda_k$, then the holomorphic function $\psi(\omega) = \ang{J^\omega f,b}$ is identically zero by the identity theorem. But $\{J^k f\}_{k \in \N_0}$ was assumed to be a frame, a contradiction.
\end{proof}

The following proposition is the key component of an alternative proof of Theorem \ref{T:Density-redundancy}. Its proof is relegated to Appendix \ref{AP:D}.

\begin{proposition}\label{prop:compact-synthesis}
Assume $\sigma_p(J)\subset(0,1)$ and $n_*<\infty$.  Let $r>0$, and let $\lambda_k/(rk)\to1$.  Denote by $\Phi_\Lambda$ and $\Phi_r$ the synthesis operators of
\[
       \{J^{\lambda_k}f\}_{k\ge0}
       \quad\text{and}\quad
       \{J^{rk}f\}_{k\ge0},
\]
respectively.  If both systems are Bessel, then $\Phi_\Lambda-\Phi_r$
is a compact operator.
\end{proposition}

\begin{proof}[Proof of Theorem \ref{T:Density-redundancy}]
We will only cover the case of $0\notin \sigma(J)$. The extension to the non-invertible case is the same as in the original proof.

Assume $0<L<\infty$ and put $r=L^{-1}$. As observed in Remark \ref{R:LBD}, condition \eqref{E:limit_assump} implies $\lambda_k/(rk)\to 1$.  Theorem~\ref{T:redundancy} says that
\[
       \{J^{rk}f:k\ge0\}
\]
is a frame.  Proposition~\ref{P:besselness} says that the system $\{J^{\lambda_k}f\}$ is Bessel.  Proposition~\ref{prop:compact-synthesis} applies and shows that its synthesis operator is a compact perturbation of the synthesis operator of the reference frame.

By \cite[Theorem~22.2.1]{Christensen16}, a Bessel sequence whose synthesis operator is a compact perturbation of a frame synthesis operator is a frame sequence.  Finally, $\lambda_k\sim rk$ implies the M\"untz divergence condition, and Proposition~\ref{P:muntz_sufficient} says that $\{J^{\lambda_k}f\}$ is complete.  A complete frame sequence is a frame for the whole space.  This proves the result.
\end{proof}

\subsection{On the M\"untz Condition}

Proposition \ref{P:muntz_sufficient} shows that the M\"untz condition is sufficient for completeness of sampled orbits $\{J^{\lambda_k}g\}_{k \in \N_0}$. However, we show here that the M\"untz condition is never necessary for completeness. 

This result will require a recent theorem of Pu-Ting Yu. 

\begin{theorem}[\protect{\cite[Theorem~3.3]{Yu2026}}]\label{T:pu-ting}
    If $\{f_k\}_{k \in \N_0}$ is a frame for $\ell^2(\N_0)$, then there exists a subsequence $\{\lambda_k\}_{k \in \N_0}$ of integers such that $\{f_{\lambda_k}/\|f_{\lambda_k}\|\}_{k \in \N_0}$ is a frame. 
\end{theorem}

\begin{proposition}\label{P:Yu}
    Let $\{J^kg\}_{k \in \N_0}$ be a block diagonal Carleson frame with $\sigma_p(J) \subset (0,1)$. Then there exists a sequence $\Lambda = \{\lambda_k\}_{k \in \N_0} \subset \N_0$ such that $\{J^{\lambda_k}g\}_{k \in \N_0}$ is complete and 
    \[
    \sum_{k \in \N_0} \frac{\lambda_k}{\lambda_k^2+1} < \infty.
    \]
\end{proposition}

\begin{proof}
    Let $\{\lambda_k\}_{k \in \N_0} \subset \N_0$ be the subsequence guaranteed by Theorem \ref{T:pu-ting}. Then $\{J^{\lambda_k}g\}_{k \in \N_0}$ is complete in $\ell^2(\N_0)$. By the resolved Feichtinger conjecture, the normalized frame $\{J^{\lambda_k}g/\|J^{\lambda_k}g\|\}_{k \in \N_0}$ is a union of finitely many Riesz sequences. If $\{\lambda_k\}_{k\in \N_0}$ satisfies the M\"untz condition
    \[
    \sum_{k\in \N_0} \frac{\lambda_k}{\lambda_k^2+1} = \infty,
    \]
    then one such Riesz subsequence $\{J^{\mu_k}g/\|J^{\mu_k}g\|\}_{k \in \N_0}$ must also satisfy 
    \[
    \sum_{k \in \N_0} \frac{\mu_k}{\mu_k^2+1} = \infty.
    \]
    But by Proposition \ref{P:muntz_sufficient}, the sequence $\{J^{\mu_k}g\}_{k \in \N_0}$ is complete but not minimal. Therefore, $\{J^{\mu_k}g/\|J^{\mu_k}g\|\}_{k \in \N_0}$ is not minimal, a contradiction. So, $\{\lambda_k\}_{k \in \N_0}$ must fail the M\"untz condition. 
\end{proof}

\begin{remark}\label{R:spectrum-tailored}
\emph{
The M\"untz condition is universal: it guarantees completeness for every positive Carleson spectrum with block sizes bounded by $n_*$.  For a fixed spectrum, it is generally stronger than necessary.}
\emph{
Let $B$ be the Blaschke product with zeros $z_i$ of multiplicity $n_i$, and let
\[
       K_B=H^2\ominus BH^2.
\]
For integer exponents $\Lambda\subset\Nzero$, the synthesis map in the proof of Proposition~\ref{T:CV frames} has kernel $BH^2$ and induces an isomorphism $K_B\to\ell^2$.  Therefore,
\begin{equation}\label{eq:model-uniqueness}
\begin{aligned}
 &\{J^\lambda g:\lambda\in\Lambda\}\text{ is complete}\\
 &\qquad\Longleftrightarrow
 K_B\cap\bigl\{F\in H^2:\widehat F(\lambda)=0
 \text{ for all }\lambda\in\Lambda\bigr\}=\{0\}.
\end{aligned}
\end{equation}
Equivalently, the coefficient restriction map
\[
       F\longmapsto(\widehat F(\lambda))_{\lambda\in\Lambda}
\]
need only be injective on the particular model space $K_B$.  This can hold for sets with convergent M\"untz sum as illustrated by Proposition \ref{P:Yu}.
}
\end{remark}

The paper \cite{CHPS24} which originally explored the redundancy of Carleson frames ended with a proposition reassuring the reader that, despite the excessive redundancy of Carleson frames, every frame must contain an infinite subsequence which is not itself a frame. We end this paper on a similar note (we mention \cite{INT25} for a slightly weaker result). 

\begin{proposition}
    Let $\{f_k\}_{k \in \N_0}$ be a frame for $\ell^2(\N_0)$. Then there exists an infinite subsequence which is incomplete.
\end{proposition}

\begin{proof}
    Let $\{f_{\lambda_k}/\|f_{\lambda_k}\|\}_{k \in \N_0}$ be the frame guaranteed by Theorem \ref{T:pu-ting}. Then this frame is a union of finitely many Riesz sequences. Let $\{f_{\mu_k}/\|f_{\mu_k}\|\}_{k \in \N_0}$ be one such infinite sequence. If this Riesz sequence is incomplete, then so is $\{f_{\mu_k}\}_{k \in \N_0}$. If $\{f_{\mu_k}/\|f_{\mu_k}\|\}_{k \in \N_0}$ is a Riesz basis, then $\{f_{\mu_k}/\|f_{\mu_k}\|\}_{k =1}^\infty$ is incomplete, and so is $\{f_{\mu_k}\}_{k =1}^\infty$.
\end{proof}

\section*{Acknowledgments}  
We wish to thank Ole Christensen and Marzieh Hasannasab for bringing the subject of block diagonal generating matrices to our attention and for further useful discussions. Special thanks are in order to Pu-Ting Yu for the idea behind the proof of Proposition \ref{P:Yu}. Finally, the authors acknowledge the use of generative AI for assistance in reviewing and improving this manuscript. 

\appendix
\numberwithin{equation}{section}
\section{Change of Basis Matrices for Powers of Jordan Blocks}\label{AP:A}
Before proving Proposition \ref{P:Bound_COB}, let us first note an easy linear algebra fact which allows us to avoid direct estimates.
\begin{lemma}\label{L:boundedness}
    Let $\{L_i\}$ be a sequence of $n\times n$ invertible matrices, and $\{U_i\}$ a sequence of $n\times n$ unitary matrices. If the matrices $U_iL_i$ converge to some invertible matrix $L$, then $\sup_i\{\|L_i\|,\|L_i^{-1}\|\} < \infty$.  
\end{lemma}

\begin{proof}
    If $U_i = I$, the $n\times n$ identity, for all $i$, then the lemma is clear since matrix inversion and taking norms are continuous. The general statement follows from this since $\|L_i\| = \|U_iL_i\|$. 
\end{proof}

\begin{proof}[Proof of Proposition \ref{P:Bound_COB}]
For each $i$, write
\[
J_{n_i}(z_i)=z_iI+(1-|z_i|^2)M_{n_i},
\]
where $M_{n_i}$ has ones on the subdiagonal and zeros elsewhere. Then
\[
J_{n_i}(z_i)^r=z_i^rI+S_i,
\]
where
\[
S_i=\sum_{\ell=1}^{n_i-1}\binom{r}{\ell}z_i^{r-\ell}(1-|z_i|^2)^\ell M_{n_i}^\ell.
\]

We choose a Jordan chain for $J_{n_i}(z_i)^r$ as follows. Let $v_{i,0}=\delta_{i,0}$, and for $0\leq j\leq n_i-2$, define
\[
v_{i,j+1}=\frac{1}{1-|z_i^r|^2}S_iv_{i,j}.
\]
Then
\[
J_{n_i}(z_i)^rv_{i,j}
=
(z_i^rI+S_i)v_{i,j}
=
z_i^rv_{i,j}+(1-|z_i|^{2r})v_{i,j+1}.
\]
Thus, $\{v_{i,0},\ldots,v_{i,n_i-1}\}$ is a Jordan chain for $J_{n_i}(z_i^r)$.

Set $L_i=[v_{i,0},\ldots,v_{i,n_i-1}]$. By construction, $ J_{n_i}(z_i)^r=L_iJ_{n_i}(z_i^r)L_i^{-1}.$ It remains to show that $\sup_i\{\|L_i\|,\|L_i^{-1}\|\}<\infty.$

We first compute
\begin{align*}
\frac{1}{1-|z_i|^{2r}}S_i
&=
\sum_{\ell=1}^{n_i-1}\binom{r}{\ell}z_i^{r-\ell}
\frac{(1-|z_i|^2)^\ell}{1-|z_i|^{2r}}M_{n_i}^\ell \\
&=
\frac{rz_i^{r-1}(1-|z_i|^2)}{1-|z_i|^{2r}}M_{n_i}
+
(1-|z_i|^2)Q_i,
\end{align*}
where $Q_i$ is strictly lower triangular and zero on the subdiagonal. Since $\sup\limits_i n_i<\infty$, the matrices $Q_i$ are entrywise uniformly bounded above and below.

Let \(r>0\). For each $i$, write
\[
z_i=\rho_i\zeta_i,
\qquad
\rho_i=|z_i|,
\qquad
|\zeta_i|=1.
\]
Define
\[
c_i
=
\frac{r z_i^{r-1}(1-|z_i|^2)}{1-|z_i|^{2r}}.
\]
Then
\[
c_i
=
\zeta_i^{r-1}
\frac{r\rho_i^{r-1}(1-\rho_i^2)}
     {1-\rho_i^{2r}}.
\]
Since \(\{z_i\}_{i\in\mathbb N_0}\) is a Carleson sequence, we have
\(\rho_i\to 1\). Therefore
\[
\frac{r\rho_i^{r-1}(1-\rho_i^2)}
     {1-\rho_i^{2r}}
\longrightarrow 1.
\]

Now, by induction on $j$,
\[
v_{i,j}
=
c_i^{j}\delta_{i,j}
+
(1-\rho_i^2)q_{i,j},
\]
where $q_{i,j}$ is supported on the coordinates $\{j+1,\ldots,n_i-1\}$ and is entrywise uniformly bounded for \(0\leq j\leq n_i-1\). Therefore
\(L_i\) is lower triangular with diagonal entries
\[
1,c_i,c_i^2,\ldots,c_i^{n_i-1},
\]
and every off-diagonal entry is bounded by a constant multiple of
\(1-\rho_i^2\).

Suppose first that \(n_i=K\) for all \(i\). Let
\[
U_i
=
\operatorname{diag}
\left(
\left(\zeta_i^{1-r}\right)^{j-1}
\right)_{j=1}^{K}.
\]
Then each \(U_i\) is unitary. Moreover, \(U_iL_i\) is lower triangular with
diagonal entries
\[
\left(
\frac{r\rho_i^{r-1}(1-\rho_i^2)}
     {1-\rho_i^{2r}}
\right)^{j-1},
\qquad
1\leq j\leq K,
\]
and every off-diagonal entry is bounded by a constant multiple of
\(1-\rho_i^2\). Hence
\[
U_iL_i\to I_K.
\]
By Lemma \ref{L:boundedness},
\[
\sup_i\{\|L_i\|,\|L_i^{-1}\|\}<\infty.
\]

Finally, suppose the block sizes are not all equal. Let
\[
K=\sup_i n_i<\infty.
\]
For each \(1\leq m\leq K\), apply the preceding argument to the subsequence
of indices satisfying \(n_i=m\). This gives a uniform bound on
\(\|L_i\|\) and \(\|L_i^{-1}\|\) along each such subsequence. Since there
are only finitely many possible block sizes, taking the maximum over
\(1\leq m\leq K\) gives
\[
\sup_i\{\|L_i\|,\|L_i^{-1}\|\}<\infty.
\]
\end{proof}

\section{From Outer Frames to Frames}\label{AP:B} Now we prove Proposition \ref{P:outer_frame_to_frame}. We will first single out an easy interpolation lemma. 

\begin{lemma}\label{L:construction}
Let $\{z_i\}_{i\in\N_0}$ be a strictly increasing sequence in $[0,1)$
with $z_i\to1$. For every $0\le\ell\le m$ and every
$a_0,\ldots,a_\ell\in\R$, there exists $f\in C^m[0,1]$ such that, for
$0\le n\le m$,
\[
f^{(n)}(z_i)=
\begin{cases}
a_n, & i=0 \text{ and } 0\le n\le\ell,\\
0,   & i\ge1 \text{ or } \ell<n\le m.
\end{cases}
\]
\end{lemma}

\begin{proof}
    The function may be chosen of the form 
    \[
    f(x) =  g(x)\sum_{n = 0}^\ell \frac{a_n(x-z_0)^n}{n!}
    \]
    for a suitable smooth cutoff function $g$.
\end{proof}

\begin{proof}[Proof of Proposition \ref{P:outer_frame_to_frame}]
     Let $P_M$ denote orthogonal projection onto $W_M$. Then, by assumption $\{P_M J^{\lambda_k} f\}_{k \in \N_0}$ is a frame for $W_M$. By Proposition \ref{P:besselness}, we may choose $B \geq A > 0$ to be frame bounds for both $\{J^k f\}_{k \in \N_0}$ and $\{P_{M}J^{\lambda_k} f\}_{k \in \N_0}$ so that $B$ is also a Bessel bound for $\{P_{M-1}J^{\lambda_k} f\}_{k \in \N_0}$.
    
    By induction, it suffices to show that the synthesis operator for $\{P_{M-1} J^{\lambda_k} f\}_{k \in \N_0}$ is surjective on $W_{M-1}$. Consider a set of $n_{M-1}$ vectors in $V_{M-1} = W_{M-1} \ominus W_M$ of the form $v_\ell = \sum_{j = 0}^{n_{M-1}-1} a_{\ell,j}\delta_{M-1,j}$ for $0 \leq \ell \leq n_{M-1}-1$. Then there exists an $\epsilon > 0$ such that 
    \begin{equation}\label{E:triangular_coeffs}
    \begin{split}
    |a_{\ell,j} - 1| &< \epsilon \text{ for } j \leq \ell\\
    |a_{\ell,j}| &< \epsilon \text{ for } j > \ell
    \end{split}
    \end{equation}
    implies the $v_\ell$ are linearly independent, and hence a basis for $V_{M-1}$. We will show that such a system of vectors exists in the range of the synthesis operator of $\{P_{M-1}J^{\lambda_k} f\}$, which is enough to show surjectivity. Furthermore, it suffices to prove the claim when $n_{M-1} = n_*$ for we may always project onto a smaller generalized eigenspace of $J_{n_{*}}(z_{M-1})$.

    By $\vartheta_b(z)$ we will denote the power series
    \[
    \vartheta_b(z) = \sum_{k = 0}^\infty b_k z^{\lambda_k}.
    \]
    Fix a value $0 \leq \ell \leq n_{*}-1$ and let $\epsilon > 0$. Invoking both Proposition \ref{P:muntz_Cm} on $C^m[z_{M-1},1]$ for $m = n_{*}-1$ and Lemma \ref{L:construction}, we may choose a finitely supported vector $b_\ell \in \ell^2(\N_0)$ such that the following hold: 
    \begin{align}
        \bigg| \frac{(1-z_{M-1}^2)^{j+\frac{1}{2}}}{j!}\vartheta_{b_\ell}^{(j)}(z_{M-1}) - 1\bigg| &< \epsilon \;\text{ for } 0 \leq j \leq \ell \label{ineq:first}\\
        \bigg| \vartheta_{b_\ell}^{(j)}(z_{M-1})\bigg| &< \epsilon \; \text{ for } \ell < j \leq n_{M-1}-1 \label{ineq:second} \\
        \bigg| \vartheta^{(j)}_{b_\ell}(z_{i})\bigg| &< \epsilon \; \text{ for } i \geq M. \label{ineq:third}
    \end{align}
    Note that here we are implicitly using the hypothesis  $n_* = n_{M-1}$ and that is important for the case $M - 1 = 0$.
    
    Let $\Phi_M$ denote the synthesis operator for $\{P_M J^{\lambda_k} f\}_{k \in \N_0}$ and define 
    \[
    c_\ell = \Phi_M^*\bigg(\Phi_M\Phi_M^*\bigg)^{-1} \Phi_Mb_\ell,
    \]
    so that $\Phi_M c_\ell = \Phi_M b_\ell$. Then, using the formula for the synthesis operator derived in Section \ref{S: BDCF}, we have   
    \begin{align*}
    \|c_\ell\| &\leq \bigg\| \Phi_M^*\bigg(\Phi_M\Phi_M^*\bigg)^{-1}\bigg\|\cdot \|\Phi_M b_\ell\| \\
    &\leq \frac{\sqrt{B}}{A} \bigg(\sum_{i = M}^\infty \sum_{j = 0}^{n_i-1} \frac{(1-z_i^2)^{2j+1}}{j!^2}|\vartheta_{b_\ell}^{(j)}(z_i)|^2 \bigg)^{1/2} \\
    &\leq \frac{\sqrt{B}}{A} \epsilon \bigg(\sum_{i = M}^\infty \sum_{j = 0}^{n_i-1} \frac{(1-z_i^2)^{2j+1}}{j!^2} \bigg)^{1/2} \\
    &=:  \frac{\sqrt{B}}{A} \epsilon C,
    \end{align*}
    where we have used \eqref{ineq:third}. Note that this means 
    \[
    \|\Phi_{M-1}c_\ell\|_{\infty} \leq \| \Phi_{M-1}\|\frac{\sqrt{B}}{A}C\epsilon \leq \frac{BC}{A}\epsilon
    \]

    Now set 
    \begin{equation*}
    \begin{split}
    v_\ell := \Phi_{M-1}(b_\ell-c_\ell) &= \sum_{j = 0}^{n_{M-1}-1} \frac{(1-z_{M-1}^2)^{j+\frac{1}{2}}}{j!}(\vartheta^{(j)}_{b_\ell}(z_{M-1}) - \vartheta^{(j)}_{c_\ell}(z_{M-1}) )\delta_{M-1,j} + \Phi_M(b_\ell - c_\ell) \\
    &= \sum_{j = 0}^{n_{M-1}-1} \frac{(1-z_{M-1}^2)^{j+\frac{1}{2}}}{j!}(\vartheta^{(j)}_{b_\ell}(z_{M-1}) - \vartheta^{(j)}_{c_\ell}(z_{M-1}) )\delta_{{M-1,j}}.
    \end{split}
    \end{equation*}
    Thus, if $0 \leq j \leq \ell$, by \eqref{ineq:first} we have the estimate 
    \begin{align*}
    |\ang{v_\ell, \delta_{M-1,j}}-1| &\leq \bigg| \frac{(1-z_{M-1}^2)^{j+\frac{1}{2}}}{j!}\vartheta^{(j)}_{b_\ell}(z_{M-1})-1\bigg|  + \bigg| \frac{(1-z_{M-1}^2)^{j+\frac{1}{2}}}{j!}\vartheta^{(j)}_{c_\ell}(z_{M-1}) \bigg|  \\
    &\leq \epsilon\bigg(1 +\frac{BC}{A} \bigg).
    \end{align*}
If $\ell < j$, by \eqref{ineq:second} 
\begin{align*}
    |\ang{v_\ell, \delta_{M-1,j}}| &\leq \frac{(1-z_{M-1}^2)^{j+\frac{1}{2}}}{j!}(|\vartheta^{(j)}_{b_\ell}(z_{M-1})| + |\vartheta^{(j)}_{c_\ell}(z_{M-1})|) \\
    &\leq \epsilon\frac{(1-z_{M-1}^2)^{j+\frac{1}{2}}}{j!}(1 + \frac{BC}{A} ).
\end{align*}
Thus, if $\epsilon$ is chosen to be small enough, the discussion around \eqref{E:triangular_coeffs} proves that $v_0,...,v_{n_{M-1}-1}$ form a basis for $V_{M-1}$ and are defined to be in the range of $\Phi_{M-1}$. This completes the proof.
\end{proof}

\section{BDC Frame Operator}\label{AP:C} We supply the proof of Lemma \ref{L:frame_operator2}.

\begin{proof}[Proof of Lemma \ref{L:frame_operator2}]
For simplicity, let us assume $\sigma_p(J) \subset (0,1)$. Then 
    \begin{align*}
        \ang{S_{\N_0}\delta_{i,j}, \delta_{r,s}} &= (1-z_i^2)^{j+\frac{1}{2}}(1-z_r^2)^{s+\frac{1}{2}}\sum_{k \in \N_0} \binom{k}{j}\binom{k}{s} z_i^{k - j} z_r^{k - s} \\
        &= \frac{(1-z_i^2)^{j+\frac{1}{2}}(1-z_r^2)^{s+\frac{1}{2}}}{j!s!}  \frac{\partial^{j+s}}{\partial x^{j}\partial y^{s}} \bigg(\frac{1}{1-xy}\bigg)\bigg|_{x=z_i,y=z_r}.    
 \end{align*}
 
 By the quotient rule, 
 \[
 \frac{\partial^{j+s}}{\partial x^{j}\partial y^{s}}\bigg(\frac{1}{1-xy}\bigg) = O\bigg(\frac{1}{(1-xy)^{j+s+1}}\bigg).
 \]
 Consequently, because $j,s \leq n_{*}-1$, there exists a constant $C$ depending only on $n_{*}$ such that 
 \[
 \frac{\partial^{j+s}}{\partial x^{j}\partial y^{s}}\bigg(\frac{1}{1-xy}\bigg) \leq \frac{C}{(1-xy)^{j+s+1}}.
 \]
 holds on $x,y \in (0,1)$ for all pairs $j,s$.

 Observe the string of inequalities 
 \begin{align*}
     \sum_{i\in \N_0} \sum_{j = 0}^{n_i-1} |\ang{S_{\N_0} \delta_{i,j}, \delta_{r,s}}| &\leq \sum_{i \in \N_0}\sum_{j = 0}^{n_{*}-1} \frac{(1-z_i^2)^{j+\frac{1}{2}}(1-z_r^2)^{s+\frac{1}{2}}}{j!s!} \frac{C}{(1-z_iz_r)^{j+s+1}} \\
     &\leq C\sum_{i \in \N_0}\sum_{j = 0}^{n_{*}-1} \frac{(1-z_i^2)^{j+\frac{1}{2}}(1-z_r^2)^{s+\frac{1}{2}}}{(1-z_iz_r)^{j+s+1}} \\
    &\leq  C \sum_{i \in \N_0}\sum_{j = 0}^{n_{*}-1} \frac{2^{j+s+1} (1-z_i)^{j+\frac{1}{2}} (1-z_r)^{s+\frac{1}{2}}}{(1-z_iz_r)^{j+s+1}} \\
 \end{align*}
 where we have used that $(1-z^2) \leq 2(1-z)$ for $z \in (0,1)$. 
 
 Now, since $\{z_i\}_{i \in \N_0}$ is a positive Carleson sequence, there exists a constant $c \in (0,1)$ such that 
 \[
 \frac{1-z_{i+1}}{1-z_i} < c
 \]
 for all $i \in \N_0$ (see \cite[Lecture~VII]{NN2012}). So, if $i \geq r$, then 
 \begin{align*}
 \frac{(1-z_i)^{j+\frac{1}{2}} (1-z_r)^{s+\frac{1}{2}}}{(1-z_iz_r)^{j+s+1}} &\leq  \frac{(1-z_i)^{j+\frac{1}{2}} (1-z_r)^{s+\frac{1}{2}}}{(1-z_r)^{j+s+1}} \\
 &= \bigg(\frac{1-z_i}{1-z_r}\bigg)^{j+\frac{1}{2}}\\
 &\leq c^{(j+\frac{1}{2})(i-r)} \\
 &\leq c^{\frac{1}{2}(i-r)}.
 \end{align*}
 Similarly, if $i < r$, then  
 \[
 \frac{(1-z_i)^{j+\frac{1}{2}} (1-z_r)^{s+\frac{1}{2}}}{(1-z_iz_r)^{j+s+1}} \leq c^{\frac{1}{2}(r-i)}. 
 \]
 Setting $C_{n_{*}} = n_{*} 2^{2n_{*}-1}C$, we have 
 \begin{align*}
     C \sum_{i \in \N_0}\sum_{j = 0}^{n_{*}-1} \frac{2^{j+s+1} (1-z_i)^{j+\frac{1}{2}} (1-z_r)^{s+\frac{1}{2}}}{(1-z_iz_r)^{j+s+1}} &\leq \frac{1}{2}C_{n_{*}}\bigg(\sum_{i = 0}^{r-1} c^{\frac{1}{2}(r-i)} + \sum_{i = r}^\infty c^{\frac{1}{2}(i-r)}\bigg) \\
     &\leq \frac{C_{n_{*}}}{1 - c^{1/2}}
 \end{align*}
 and the proof is complete.
\end{proof}

\section{Compactness of the perturbation}\label{AP:D} Finally, we prove Proposition \ref{prop:compact-synthesis}.

\begin{proof}[Proof of Proposition \ref{prop:compact-synthesis}]
For a block coordinate $(i,q)$, set
\[
 u_{i,q}(t)=\binom tq(1-z_i^2)^{q+1/2}z_i^{t-q}.
\]
The $k$-th column of $K:=\Phi_\Lambda-\Phi_r$ has coordinates
\[
       u_{i,q}(\lambda_k)-u_{i,q}(rk).
\]
Let $K_N$ be the operator obtained by keeping only the first $N$ columns. Clearly, $K_N$ has finite rank, and it suffices to show that $\norm{K-K_N}\to0$.

Put
\[
       \eta_N=\sup_{k\ge N}\left|\frac{\lambda_k}{rk}-1\right|,
\]
so that $\eta_N\downarrow0$, and  for $N$ large, we have $\eta_N<1/2$.  The mean value theorem and differentiation of $\binom tq z_i^t$ give, uniformly for $q<n_*$ and for $t$ between $rk$ and $\lambda_k$,
\begin{equation}\label{eq:column-difference}
\begin{split}
 |u_{i,q}(\lambda_k)-u_{i,q}(rk)|
 \le{}& C\eta_N(1-z_i^2)^{q+1/2}z_i^{rk/2-q}\\
 &\times\big((1+rk)^q+|\log z_i|(1+rk)^{q+1}\big).
\end{split}
\end{equation}
Indeed, $|\lambda_k-rk|\le\eta_Nrk$, while $z_i^t\le z_i^{rk/2}$ on the intervening interval.

Form the matrix of the tail operator $(K-K_N)(K-K_N)^*$.  By Cauchy--Schwarz in the $k$ sum, \eqref{eq:column-difference}, and the standard estimates
\[
       \sum_{k\ge0}(1+k)^d x^k\le C_d(1-x)^{-d-1},
       \qquad
       |\log z|\asymp1-z\quad(z\uparrow1),
\]
each matrix entry is bounded by $C\eta_N^2$ times a finite sum of kernels of the form
\[
 \frac{(1-z_i^2)^{p+1/2}(1-z_j^2)^{q+1/2}}
      {(1-z_i^{r/2}z_j^{r/2})^{p+q+1}},
 \qquad p,q<n_*.
\]
An argument similar to the one in Proposition~\ref{P:besselness} gives uniform row and column sums for these kernels.  Schur's test therefore yields
\[
       \norm{(K-K_N)(K-K_N)^*}\le C\eta_N^2,
\]
and, hence, $\norm{K-K_N}\le C^{1/2}\eta_N\to0$.  Thus, $K$ is the norm limit of finite-rank operators and is compact.
\end{proof}

\bibliographystyle{siam}
\bibliography{refs1}

@misc{Yu2026,
      title={Every semi-normalized unconditional Schauder frame in Hilbert spaces contains a frame}, 
      author={Pu-Ting Yu},
      year={2026 arXiv 2602.21616},
      eprint={2602.21616},
      archivePrefix={arXiv},
      primaryClass={math.CA},
      url={https://arxiv.org/abs/2602.21616}, 
}

@article {BHR22,
    AUTHOR = {Bergelson, Vitaly and Huryn, Jake and Raghavan, Rushil},
     TITLE = {Discordant sets and ergodic {R}amsey theory},
   JOURNAL = {Involve},
  FJOURNAL = {Involve. A Journal of Mathematics},
    VOLUME = {15},
      YEAR = {2022},
    NUMBER = {1},
     PAGES = {89--130},
      ISSN = {1944-4176,1944-4184},
   MRCLASS = {05D10 (37A44)},
  MRNUMBER = {4396354},
MRREVIEWER = {Siming\ Tu},
       DOI = {10.2140/involve.2022.15.89},
       URL = {https://doi-org.proxy.library.vanderbilt.edu/10.2140/involve.2022.15.89},
}

@book {Seip04,
    AUTHOR = {Seip, Kristian},
     TITLE = {Interpolation and sampling in spaces of analytic functions},
    SERIES = {University Lecture Series},
    VOLUME = {33},
 PUBLISHER = {American Mathematical Society, Providence, RI},
      YEAR = {2004},
     PAGES = {xii+139},
      ISBN = {0-8218-3554-8},
   MRCLASS = {30E05 (30D45 30D55 46E15 46E20 47B35 94A20)},
  MRNUMBER = {2040080},
MRREVIEWER = {Richard\ Rochberg},
       DOI = {10.1090/ulect/033},
       URL = {https://doi-org.proxy.library.vanderbilt.edu/10.1090/ulect/033},
}

@article {Almira07,
    AUTHOR = {Almira, J. M.},
     TITLE = {M\"untz type theorems. {I}},
   JOURNAL = {Surv. Approx. Theory},
  FJOURNAL = {Surveys in Approximation Theory},
    VOLUME = {3},
      YEAR = {2007},
     PAGES = {152--194},
      ISSN = {1555-578X},
   MRCLASS = {41A30 (41-02)},
  MRNUMBER = {2350268},
MRREVIEWER = {R.\ A.\ Zalik},
}

@misc{GP26arXiv,
      title={Frame constructions associated with operator orbits}, 
      author={Eva A. Gallardo-Gutiérrez and Jonathan R. Partington},
      year={2026 arXiv 2605.29671},
      eprint={2605.29671},
      archivePrefix={arXiv},
      primaryClass={math.FA},
      url={https://arxiv.org/abs/2605.29671}, 
}

@book{K04t,
    AUTHOR = {Korevaar, Jacob},
     TITLE = {Tauberian theory},
    SERIES = {Grundlehren der mathematischen Wissenschaften [Fundamental
              Principles of Mathematical Sciences]},
    VOLUME = {329},
      NOTE = {A century of developments},
 PUBLISHER = {Springer-Verlag, Berlin},
      YEAR = {2004},
     PAGES = {xvi+483},
      ISBN = {3-540-21058-X},
   MRCLASS = {11M45 (11-03 11N05 40-02 40-03 40E05 42-02)},
  MRNUMBER = {2073637},
MRREVIEWER = {Angel\ V.\ Kumchev},
       DOI = {10.1007/978-3-662-10225-1},
       URL = {https://doi-org.proxy.library.vanderbilt.edu/10.1007/978-3-662-10225-1},
}

@article {ACKM26,
    AUTHOR = {Aldroubi, Akram and Cabrelli, Carlos and Krishtal, Ilya and
              Molter, Ursula},
     TITLE = {Dynamical sampling: a survey},
   JOURNAL = {Matematica},
  FJOURNAL = {La Matematica. Official Journal of the Association for Women
              in Mathematics},
    VOLUME = {5},
      YEAR = {2026},
    NUMBER = {2},
     PAGES = {Paper No. 37, 28},
      ISSN = {2730-9657},
   MRCLASS = {42C15 (47A10 65T60 94A20)},
  MRNUMBER = {5071919},
       DOI = {10.1007/s44007-026-00215-y},
       URL = {https://doi.org/10.1007/s44007-026-00215-y},
}

@article{KM25d,
title = {Demystifying Carleson frames},
journal = {Applied and Computational Harmonic Analysis},
volume = {80},
pages = {101811},
year = {2026},
issn = {1063-5203},
doi = {https://doi.org/10.1016/j.acha.2025.101811},
url = {https://www.sciencedirect.com/science/article/pii/S106352032500065X},
author = {Ilya Krishtal and Brendan Miller}
}

@article {CMPP20,
    AUTHOR = {Cabrelli, Carlos and Molter, Ursula and Paternostro, Victoria
              and Philipp, Friedrich},
     TITLE = {Dynamical sampling on finite index sets},
   JOURNAL = {J. Anal. Math.},
  FJOURNAL = {Journal d'Analyse Math\'ematique},
    VOLUME = {140},
      YEAR = {2020},
    NUMBER = {2},
     PAGES = {637--667},
      ISSN = {0021-7670,1565-8538},
   MRCLASS = {42C15 (47B15 94A20)},
  MRNUMBER = {4093918},
MRREVIEWER = {Qingyue\ Zhang},
       DOI = {10.1007/s11854-020-0099-2},
       URL = {https://doi.org/10.1007/s11854-020-0099-2},
}

@book {Garnett81,
    AUTHOR = {Garnett, John B.},
     TITLE = {Bounded analytic functions},
    SERIES = {Pure and Applied Mathematics},
    VOLUME = {96},
 PUBLISHER = {Academic Press, Inc. [Harcourt Brace Jovanovich, Publishers],
              New York-London},
      YEAR = {1981},
     PAGES = {xvi+467},
      ISBN = {0-12-276150-2},
   MRCLASS = {30D55 (30-02 46J15)},
  MRNUMBER = {628971},
MRREVIEWER = {D.\ Sarason},
}

@article{CHP20,
author = {Christensen, Ole and Hasannasab, Marzieh and Philipp, Friedrich},
title = {Frame properties of operator orbits},
journal = {Mathematische Nachrichten},
volume = {293},
number = {1},
pages = {52-66},
doi = {https://doi.org/10.1002/mana.201800344},
url = {https://onlinelibrary.wiley.com/doi/abs/10.1002/mana.201800344},
eprint = {https://onlinelibrary.wiley.com/doi/pdf/10.1002/mana.201800344},
year = {2020}
}

@article {INT25,
    AUTHOR = {Izbiakov, Ilya and Novikov, Sergey and Terekhin, Pavel},
     TITLE = {Complement property and frames in the phase retrieval problem},
      NOTE = {Translation of Funktsional. Anal. i Prilozhen. {\bf 59}
              (2025), no. 1, 18--28},
   JOURNAL = {Funct. Anal. Appl.},
  FJOURNAL = {Functional Analysis and its Applications},
    VOLUME = {59},
      YEAR = {2025},
    NUMBER = {1},
     PAGES = {11--18},
      ISSN = {0016-2663,1573-8485},
   MRCLASS = {42C15 (42C30 46B15)},
  MRNUMBER = {4902502},
       DOI = {10.1134/s1234567825010021},
       URL = {https://doi.org/10.1134/s1234567825010021},
}

@book{Christensen16,
	author = {Christensen, Ole},
	publisher = {Birkhäuser Cham},
	title = {An Introduction to Frames and Riesz Bases},
	year = {2016}}

@article {CHPS24,
    AUTHOR = {Christensen, Ole and Hasannasab, Marzieh and Philipp,
              Friedrich M. and Stoeva, Diana},
     TITLE = {The mystery of {C}arleson frames},
   JOURNAL = {Appl. Comput. Harmon. Anal.},
  FJOURNAL = {Applied and Computational Harmonic Analysis. Time-Frequency
              and Time-Scale Analysis, Wavelets, Numerical Algorithms, and
              Applications},
    VOLUME = {72},
      YEAR = {2024},
     PAGES = {Paper No. 101659, 5},
      ISSN = {1063-5203},
   MRCLASS = {42C15},
  MRNUMBER = {4735722},
       DOI = {10.1016/j.acha.2024.101659},
       URL = {https://doi.org/10.1016/j.acha.2024.101659},
}

@article{ACCMP17,
	Author = {Aldroubi, A. and Cabrelli, C. and \c{C}akmak, A. F. and Molter, U. and Petrosyan, A.},
	Doi = {10.1016/j.jfa.2016.10.027},
	Fjournal = {Journal of Functional Analysis},
	Issn = {0022-1236},
	Journal = {J. Funct. Anal.},
	Mrclass = {47A10 (42C15)},
	Mrnumber = {3579135},
	Mrreviewer = {Ilya A. Krishtal},
	Number = {3},
	Pages = {1121--1146},
	Title = {Iterative actions of normal operators},
	Url = {http://dx.doi.org/10.1016/j.jfa.2016.10.027},
	Volume = {272},
	Year = {2017}}

@article{ACMT17,
	Author = {A. Aldroubi and C. Cabrelli and U. Molter and S. Tang},
	Doi = {http://dx.doi.org/10.1016/j.acha.2015.08.014},
	Issn = {1063-5203},
	Journal = {Applied and Computational Harmonic Analysis},
	Note = {doi: 10.1016/j.acha.2015.08.014},
	Number = {3},
	Pages = {378--401},
	Title = {Dynamical sampling},
	Url = {http://www.sciencedirect.com/science/article/pii/S1063520315001177},
	Volume = {42},
	Year = {2017}}

@article {R56,
    AUTHOR = {Rubel, L. A.},
     TITLE = {Necessary and sufficient conditions for {C}arlson's theorem on
              entire functions},
   JOURNAL = {Trans. Amer. Math. Soc.},
  FJOURNAL = {Transactions of the American Mathematical Society},
    VOLUME = {83},
      YEAR = {1956},
     PAGES = {417--429},
      ISSN = {0002-9947},
   MRCLASS = {30.0X},
  MRNUMBER = {81944},
MRREVIEWER = {R. P. Boas},
       DOI = {10.2307/1992882},
       URL = {https://doi.org/10.2307/1992882},
}

@book{NN2012, title={Treatise on the shift operator: Spectral Function theory}, publisher={Springer}, author={Nikol’skii, N. K.}, year={2012}}
\end{document}